\documentclass[11pt]{amsart}
\usepackage[utf8]{inputenc}

\usepackage[usenames,dvipsnames]{xcolor}
\usepackage[draft]{pgf}
\usepackage{listings}
\usepackage{float}
\usepackage{caption}

\usepackage{pgfplots}
\pgfplotsset{compat=1.15}
\usepackage{mathrsfs}
\usetikzlibrary{arrows}

\usepackage[margin=1.2in, centering]{geometry}
\usepackage{amsmath,amssymb,latexsym,mathtools}
\usepackage{dsfont}
\usepackage{nccmath}

\usepackage{amsthm}
\usepackage{amssymb}
\newtheorem{thm}{Theorem}[section]
\newtheorem{lem}[thm]{Lemma}

\numberwithin{equation}{section}

\theoremstyle{definition}

\newtheorem{defn}[thm]{Definition}

\newtheorem*{remark}{Remark}

\newtheorem*{acknowledgement}{Acknowledgements}

\usepackage{graphicx}
\DeclareGraphicsExtensions{.pdf,.png,.jpg}
\usepackage{subfig}

\counterwithin*{equation}{section}

\newcommand{\comment}[1]{}

\newcommand{\norm}[1]{{\left\Vert#1\right\Vert}}

\newcommand{\genlegendre}[4]{%
  \genfrac{(}{)}{}{#1}{#3}{#4}%
  \if\relax\detokenize{#2}\relax\else_{\!#2}\fi
}
\newcommand{\legendre}[3][]{\genlegendre{}{#1}{#2}{#3}}

\newcommand\oneE[1][E]{\mathop{\mathds{1}^{}_{#1}}} 

\usepackage{hyperref}
\hypersetup{linktocpage}

\begin{document}

\definecolor{qqttcc}{rgb}{0.,0.2,0.8}
	\title{Distribution of similar configurations in subsets of $\mathbb{F}_q^d$}
	\author{Firdavs Rakhmonov}
	
	\date{\today}
	\address{Department of Mathematics, University of Rochester, Rochester, NY, USA}
	\email{frakhmon@ur.rochester.edu}
	\maketitle

\begin{abstract} 
	Let $\mathbb{F}_q$ be a finite field of order $q$ and $E$ be a set in $\mathbb{F}_q^d$. The distance set of $E$ is defined by $\Delta(E):=\{\lVert x-y \rVert :x,y\in E\}$, where $\lVert \alpha \rVert=\alpha_1^2+\dots+\alpha_d^2$. Iosevich, Koh and Parshall (2018) proved that if $d\geq 2$ is even and $|E|\geq 9q^{d/2}$, then $$\mathbb{F}_q= \frac{\Delta(E)}{\Delta(E)}=\left\{\frac{a}{b}: a\in \Delta(E),\ b\in \Delta(E)\setminus\{0\} \right\}.$$ In other words, for each $r\in \mathbb{F}_q^*$ there exist $(x,y)\in E^2$ and $(x',y')\in E^2$ such that $\lVert x-y\rVert\neq0$ and $\lVert x'-y' \rVert=r\lVert x-y\rVert$.

    \smallskip
	
	Geometrically, this means that if the size of $E$ is large, then for any given $r \in \mathbb{F}_q^*$ we can find a pair of edges in the complete graph $K_{|E|}$ with vertex set $E$ such that one of them is dilated by $r\in \mathbb{F}_q^*$ with respect to the other. A natural question arises whether it is possible to generalize this result to arbitrary subgraphs of $K_{|E|}$ with vertex set $E$ and this is the goal of this paper.

    \smallskip
    
    In this paper, we solve this problem for $k$-paths $(k\geq 2)$, simplexes and 4-cycles. We are using a mix of tools from different areas such as enumerative combinatorics, group actions and Turán type theorems. 
\end{abstract} 

\section{Introduction}

Many problems in discrete geometry ask whether certain structure exists in a set of sufficiently large size. The most renowned result of this type belongs to Szemerédi ~\cite{MR369312} which claims that a subset of natural numbers with positive upper density contains arbitrarily long arithmetic progressions. 

\smallskip

The Erdős distinct distances problem asks for the smallest possible number of distinct distances determined by a finite subset of $\mathbb{R}^d$, $d\geq 2$. In ~\cite{MR3272924}, Guth and Katz solved the Erdős distance problem for $d=2$. For $d\geq 3$ the conjecture remains open with the best results due to Solymosi and Vu (\cite{MR2065269}). They proved that the number of distinct distances in a well-distributed set of $n$ points in $\mathbb{R}^d$ is $\Omega(n^{2/d-1/d^2})$ which is close to the best known upper bound $O(n^{2/d})$ and they have a further improvement $\Omega(n^{0.5794})$ in the case $d=3$. The continuous analog of this problem is called Falconer's conjecture, which asks for the smallest Hausdorff dimension of a subset $E$ of $\mathbb{R}^d$ ($d\geq 2$) such that the Lebesgue measure of the distance set \begin{equation*}
    \Delta(E):=\{|x-y|: x,y\in E\}\end{equation*} is positive.

\smallskip

The Erdős-Falconer distance problem in vector spaces over finite fields asks for the smallest possible size of $$\Delta(E):=\{\lVert x-y\rVert: x,y\in E\},$$ where  $\lVert \alpha\rVert:=\alpha_1^2+\dots+\alpha_d^2$ for $\alpha=(\alpha_1,\dots,\alpha_d)$ and $E\subset \mathbb{F}_q^d$, $d\geq 2$. This problem was introduced by Bourgain, Katz and Tao in ~\cite{MR2053599}. Here $\mathbb{F}_q$ denotes the finite field with $q$ elements and $\mathbb{F}_q^d$ is the $d$-dimensional vector space over this field.

\smallskip

In ~\cite{MR2336319}, Iosevich and Rudnev proved that if $E\subset \mathbb{F}_q^d$, $d\geq 2$, with $|E|>2q^{\frac{d+1}{2}}$, then $\Delta(E)=\mathbb{F}_q$. Moreover, they proved that one cannot in general obtain $cq$ distances if $|E|\ll q^{\frac{d}{2}}$. In ~\cite{MR2775806}, Hart, Iosevich, Koh and Rudnev proved that the size of the critical exponent can be increased to $\frac{d+1}{2}$ in odd dimensions, thus showing that the result in ~\cite{MR2336319} is best possible in this setting. In ~\cite{MR3592595}, Bennett et al. proved that if $d=2$ and $|E|\gg q^{\frac{4}{3}}$, then $|\Delta(E)|>cq$. It is reasonable to conjecture that if $|E|\gg q^{\frac{d}{2}},$ then $|\Delta(E)|>\frac{q}{2},$ but this conjecture still remains open. The exponent $\frac{d}{2}$ cannot be improved. Indeed, let $q=p^2$, where $p$ is a prime and let $E=\mathbb{F}_p^d\subset \mathbb{F}_q^d$. Then $|E|=q^{\frac{d}{2}}$, yet $\Delta(E)=\mathbb{F}_p$. In ~\cite{MR2775806}, Hart et al. proved that if $q$ is a prime and $d\geq 4$, the sharpness of $\frac{d}{2}$ can be demonstrated using Lagrangian subspaces.

\smallskip

Iosevich, Koh and Parshall (see ~\cite{MR3959878}) showed that if $d\geq 2$ is even and $E\subset \mathbb{F}_q^d$ with $|E|\geq 9q^{\frac{d}{2}}$, then
\begin{equation}
\label{containmentevencase}
    \mathbb{F}_q=\dfrac{\Delta(E)}{\Delta(E)}=\left\{\frac{a}{b}: a\in \Delta(E),\ b\in \Delta(E)\setminus\{0\}\right\}.
\end{equation}

If the dimension $d$ is odd and $E\subset \mathbb{F}_q^d$ with $|E|\geq 6 q^{\frac{d}{2}}$, then 
\begin{equation}
\label{containmentoddcase}
    \left(\mathbb{F}_q\right)^2\subset \dfrac{\Delta(E)}{\Delta(E)}, 
\end{equation}
where $\left(\mathbb{F}_q\right)^2:=\left\{a^2:a\in \mathbb{F}_q \right\}$ is the set of quadratic residues in $\mathbb{F}_q$.

\medskip

We shall write $\mathbb F_q^*$ for the set of all non-zero elements in $\mathbb F_q.$                          
\smallskip

As the main idea to deduce \eqref{containmentevencase} and \eqref{containmentoddcase}, the authors \cite{MR3959878} first observed that for any $r\in \mathbb F_q^*,$ we have 
\begin{equation*}
    r\in \frac{\Delta(E)}{\Delta(E)} \quad \textup{if} \quad  \sum_{t\in \mathbb F_q} \nu(t) \nu(rt) >\nu^2(0),    
\end{equation*}
where $\nu(t)\coloneqq \{(x,y)\in E\times E: \lVert x-y\rVert=t\}$.

\medskip

Next, using the discrete Fourier analysis, they  estimated  a lower bound of $\sum_{t\in \mathbb F_q} \nu(t) \nu(rt)$ and an upper bound of $\nu^2(0).$
Finally,  the required size condition on the sets $E$ was obtained by comparing them. Although the method of the proof led to the optimal threshold result on the problem for the quotient set of the distance set, it has two drawbacks below, as mentioned by Pham \cite{Pham}.

\smallskip

\begin{itemize}
\item The proof is too sophisticated and requires large amount of computation.

\smallskip

\item It is not clear from the proof that  how many quadruples $(x,y,z,w)\in E^4$ contribute to producing each element $r\in \mathbb F_q^*$ such that $\frac{||x-y||}{||z-w||}=r,$ namely, $r\in \frac{\Delta(E)}{\Delta(E)}.$
\end{itemize}

As a way to overcome the above issues, Pham \cite{Pham} utilized the machinery of group actions in two dimensions and obtained   a lower bound of $V(r)$ for any square number $r$ in $\mathbb F_q^*,$ 
where $V(r)\coloneqq \{(x,y,z,w)\in E^4: \lVert x-y \rVert / \lVert z-w \rVert=r\}$.

\smallskip

As a consequence, he provided a short proof to deduce the following lower bound of $V(r)$ in two dimensions.

\smallskip

\begin{thm} [Theorem 1.2, \cite{Pham}]\label{Phthm} Let $E\subset \mathbb F_q^2$ and suppose that $|E|\ge C q$ with $q\equiv 3 \pmod{4}.$ Then, for any non-zero square number $r$ in $\mathbb F_q^*,$  we have
$$ V(r)\ge  \frac{c|E|^4}{q}.$$
In particular,  we have $\dfrac{\Delta(E)}{\Delta(E)} \supseteq (\mathbb F_q)^2:=\{a^2: a\in \mathbb F_q\}.$
\end{thm}

Pham's approach, based on the group action, is  powerful in the sense that it gives a relatively simple proof and an information about a lower bound of $V(r)$. However, his result, Theorem \ref{Phthm}, is limited to two dimensions with $-1$ square in $\mathbb F_q,$ and it gives us no information about $V(r)$ for a non-square $r$ in $\mathbb F_q^*.$ 

\smallskip

In \cite{RFZ2}, Iosevich, Koh and Rakhmonov improved and extended Pham's result to all dimensions for arbitrary finite fields. In particular, they worked with the non-degenerate quadratic distances which generalize the usual distance.

\smallskip

Equality \eqref{containmentevencase} immediately implies that for each $r\in \mathbb{F}_q^*$ one can find $(x,y)\in E^2$ and $(x',y')\in E^2$ such that $\lVert x-y \rVert\neq 0$ and $\lVert x'-y' \rVert=r\lVert x-y \rVert$. In other words, if $r\in \mathbb{F}_q^*$ and $E\subset \mathbb{F}_q^d$ with $|E|\geq 9 q^{\frac{d}{2}}$, then one can find a pair of edges in the complete graph $K_{|E|}$ with vertex set $E$ such that one of them is dilated by $r$ with respect to the other.

\begin{figure}[htbp]
\centering
\usetikzlibrary{arrows}

\definecolor{qqqqff}{rgb}{0.,0.,1.}
\definecolor{ududff}{rgb}{0.30196078431372547,0.30196078431372547,1.}

\begin{tikzpicture}[scale=0.75][line cap=round,line join=round,>=triangle 45,x=1.0cm,y=1.0cm]

\clip(-1.,-1) rectangle (16.,1.8);
\draw [line width=1.pt,color=qqqqff] (0.96,-0.8)-- (4.46,0.82);
\draw [line width=1.pt,color=qqqqff] (7.9,-0.84)-- (14.2,0.76);

\draw [fill=ududff] (0.96,-0.8) circle (2pt);
\draw[color=black] (0.65,-0.83) node {\scalebox{1.2}{$x$}};
\draw [fill=ududff] (4.46,0.82) circle (2pt);
\draw[color=black] (4.8,0.8) node {\scalebox{1.2}{$y$}};
\draw [fill=ududff] (7.9,-0.84) circle (2pt);
\draw[color=black] (7.6,-0.75) node {\scalebox{1.2}{$x'$}};
\draw [fill=ududff] (14.2,0.76) circle (2pt);
\draw[color=black] (14.6,0.8) node {\scalebox{1.2}{$y'$}};

\draw[color=qqqqff] (2.5,0.4) node {\scalebox{1.2}{$\alpha$}};

\draw[color=qqqqff] (11,0.5) node {\scalebox{1.2}{$r\alpha$}};

\end{tikzpicture}
\label{1-chains}
\caption{Geometric interpretation of Iosevich-Koh-Parshall result.}
\end{figure}
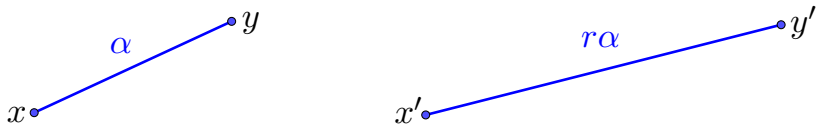

A natural question arises whether it is possible to generalize this result to other subgraphs of the complete graph $K_{|E|}$ with vertex set $E$ and this is the goal of this paper. In this paper, we are about to generalize this result for $k$-paths ($k\geq 2$), 4-cycles and simplexes. 

\smallskip

For clarity, let's discuss the case of 2-paths. Firstly, we assume that $p$ is a prime such that $p\equiv 3 \pmod 4$ since it implies $\lVert x\rVert=0$ iff $x=(0,0)$. Then we prove that for each $r\in \mathbb{F}_p^*$ and $E\subset \mathbb{F}_p^2$ with $|E|\gg p$ there are two copies of 2-path with vertices in $E$, i.e. $(x_1,x_2,x_3)\in E^3$ and $(y_1,y_2,y_3)\in E^3$ such that
\begin{equation}
\label{nondegeneratecondition}
    x_i\neq x_j,\ y_i\neq y_j\quad \text{for}\quad i\neq j
\end{equation} and
\begin{equation}
\label{dilation}
    \lVert y_1-y_2\rVert=r\lVert x_1-x_2\rVert,\quad \lVert y_2-y_3\rVert=r\lVert x_2-x_3\rVert.
\end{equation}

We notice that condition $\eqref{nondegeneratecondition}$ simply means that 2-paths $(x_1,x_2,x_3)$ and $(y_1,y_2,y_3)$ are nondegenerate, i.e. $\lVert x_i-x_j\rVert\neq 0$ and $\lVert y_i-y_j\rVert\neq 0$ for $i\neq j$ which easily follows from the fact that $x_i\neq x_j,\ y_i\neq y_j$ for $i\neq j$ and $p\equiv 3 \pmod 4$.
Condition $\eqref{dilation}$ says that the 2-path $(y_1,y_2,y_3)$ is dilated by $r$ with respect to $(x_1,x_2,x_3)$ (see Figure \ref{2chainsfigure}). 

\smallskip

Likewise, we generalize Iosevich-Koh-Parshall result for 4-cycles, simplexes and $k$-paths ($k\geq 3$). 

\smallskip

Now we pause to state some basic definitions which we need throughout the paper.

\begin{defn} 
Let $E$ be a set in $\mathbb{F}_p^2$ and $r\in \mathbb{F}_p^{*}$. We say that $(x_1,\dots,x_{k+1})\in E^{k+1}$, $(y_1,\dots,y_{k+1})\in E^{k+1}$ is a pair of $k$-paths in $E$ with dilation ratio $r$ if $\lVert y_i-y_{i+1}\rVert=r\lVert x_i-x_{i+1}\rVert$ for $i\in [k]$ and $x_i\neq x_j$, $y_i\neq y_j$  for $1\leq i<j\leq k+1$.
\end{defn}

\begin{defn}
Let $E$ be a set in $\mathbb{F}_p^2$, $r\in \mathbb{F}_p^{*}$ and $k\geq 3$. We say that $(x_1,\dots,x_{k})\in E^{k}$, $(y_1,\dots,y_{k})\in E^{k}$ is a pair of $k$-cycles in $E$ with dilation ratio $r$ if $\lVert y_i-y_{i+1}\rVert=r\lVert x_i-x_{i+1}\rVert$ for $i\in [k-1]$, $\lVert y_k-y_1\rVert=r\lVert x_k-x_1\rVert$  and $x_i\neq x_j$, $y_i\neq y_j$ for $1\leq i<j\leq k$.
\end{defn}

\begin{defn}
Let $E$ be a set in $\mathbb{F}_p^d$, $r\in \mathbb{F}_p^{*}$ and $d\geq 2$. We say that $(x_1,\dots,x_{d+1})\in E^{d+1}$, $(y_1,\dots,y_{d+1})\in E^{d+1}$ is a pair of $d$-simplexes in $E$ with dilation ratio $r$ if $\lVert y_i-y_{j}\rVert=r\lVert x_i-x_{j}\rVert$ for $i,j\in [d+1]$ and $x_i\neq x_j$,\ $y_i\neq y_j$  for $1\leq i<j\leq d+1$.
\end{defn}

\smallskip

Now we proceed to the formulation of the problem. Let $r\in \mathbb{F}_p^{*}$ and $E\subset \mathbb{F}_p^d$ $(d\geq 2)$. One can ask how large $E$ needs to be to guarantee the existence of a pair of $k$-paths ($k$-cycles or $d$-simplexes) in $E$ with dilation ratio $r$. We solve this problem for 2-paths using refined combinatorial and graph-theoretic methods and using the arithmetic structure of spheres in $\mathbb{F}_p^d$ we extend this result to 4-cycles. Then using the arithmetic structure of orthogonal groups over finite fields and tools from group actions we prove the result for 2-simplexes and then generalize it for $d$-simplexes. In this setting, the study of the distance set problem reduces to the investigation $\mathrm{O}_d(\mathbb{F}_p)$ the orthogonal group of $d\times d$ matrices with entries in $\mathbb{F}_p$ and the $L^d$,\ $L^{d+1}$-norms of the counting function $$\lambda_{r,\theta}(z):=|\{(u,v)\in E^2: u-\sqrt{r}\theta v=z\}|,$$ where $\theta\in \mathrm{O}_d(\mathbb{F}_p)$, $z\in \mathbb{F}_p^d$ and $r\in (\mathbb{F}_p)^2\setminus \{0\}$ the set of nonzero quadratic residues in $\mathbb{F}_p$. With this setup, our main results are following:

\

The first theorem asserts the existence of a pair of 2-paths  with dilation ratio $r\in \mathbb{F}_p^{*}$.

\begin{thm}
\label{2chainsthm}
    If $r\in \mathbb{F}_p^{*}$, $p$ is a prime such that $p \equiv 3 \pmod 4$ and $E\subset \mathbb{F}_p^2$ with $|E|>(\sqrt{3}+1)p$, then
    \begin{equation*}
        \left|\left\{(x_1,x_2,x_3,y_1,y_2,y_3)\in E^6: \begin{array}{l}\lVert y_i-y_{i+1}\rVert=r\lVert x_i-x_{i+1}\rVert,\ i\in [2],\\
        x_i\neq x_j,\ y_i\neq y_j,\ i\neq j \end{array}\right\}\right|>0.\\
    \end{equation*}
\end{thm}

\begin{figure}[htbp]
\centering
\usetikzlibrary{arrows}

\definecolor{qqqqff}{rgb}{0.,0.,1.}
\definecolor{ududff}{rgb}{0.30196078431372547,0.30196078431372547,1.}

\begin{tikzpicture}[line cap=round,line join=round,>=triangle 45,x=1.0cm,y=1.0cm]

\draw [line width=1.pt,color=qqqqff] (6.002541750273515,0.48657696459304567)-- (7.973458602964373,1.3740658397813257);
\draw [line width=1.pt,color=qqqqff] (7.973458602964373,1.3740658397813257)-- (9.091464069110648,0.5096286236888452);
\draw [line width=1.pt,color=qqqqff] (3.997047408938958,0.3136895213745496)-- (3.,2.);
\draw [line width=1.pt,color=qqqqff] (0.,0.)-- (3.,2.);

\draw [fill=ududff] (0.,0.) circle (2pt);
\draw[color=black] (-0.35,0) node {\scalebox{1.1}{$x_1$}};

\draw [fill=ududff] (3.997047408938958,0.3136895213745496) circle (2.0pt);
\draw[color=black] (4.35,0.3) node {\scalebox{1.1}{$x_3$}};

\draw [fill=ududff] (6.002541750273515,0.48657696459304567) circle (2.0pt);
\draw[color=black] (5.7,0.5) node {\scalebox{1.1}{$y_1$}};

\draw [fill=ududff] (7.973458602964373,1.3740658397813257) circle (2.0pt);
\draw[color=black] (8,1.65) node {\scalebox{1.1}{$y_2$}};

\draw[color=qqqqff] (6.9,1.17) node {\scalebox{1.2}{$r\alpha$}};
\draw [fill=ududff] (9.091464069110648,0.5096286236888452) circle (2.0pt);
\draw[color=black] (9.5,0.5) node {\scalebox{1.1}{$y_3$}};
\draw[color=qqqqff] (8.8,1.2) node {\scalebox{1.2}{$r\beta$}};
\draw [fill=ududff] (3.,2.) circle (2.0pt);
\draw[color=black] (3.05,2.25) node {\scalebox{1.1}{$x_2$}};
\draw[color=qqqqff] (3.8,1.33) node {\scalebox{1.2}{$\beta$}};
\draw[color=qqqqff] (1.5,1.27) node {\scalebox{1.2}{$\alpha$}};
\end{tikzpicture}
\label{2-chains}
\caption{Pair of 2-paths with dilation ratio $r\in \mathbb{F}_p^{*}$. }
\label{2chainsfigure}
\end{figure}
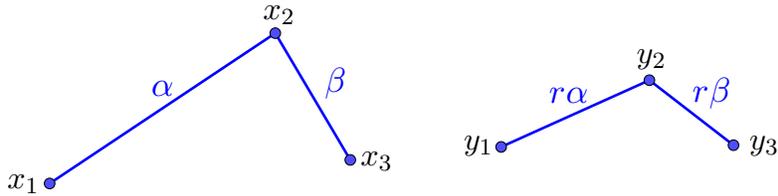

\smallskip

Once we prove this, we proceed to the case of 4-cycles.

\begin{thm}
\label{4cyclesthm}
    
    If $r\in \mathbb{F}_p^{*}$, $p$ is a prime such that $p \equiv 3 \pmod 4$ and $E\subset \mathbb{F}_p^2$ with $|E|>4\sqrt{3}p^{\frac{3}{2}}$, then
    \begin{equation}
        \label{eq1.1}
            \left|\left\{(x_1,x_2,x_3,x_4,y_1,y_2,y_3,y_4)\in E^8: \begin{array}{l}\lVert y_i-y_{i+1}\rVert=r\lVert x_i-x_{i+1}\rVert,\ i\in [3],\\
            \lVert y_4-y_{1}\rVert=r\lVert x_4-x_{1}\rVert,\\
x_i\neq x_j,\ y_i\neq y_j,\ i\neq j \end{array}\right\}\right|>0\\
    \end{equation}
\end{thm}

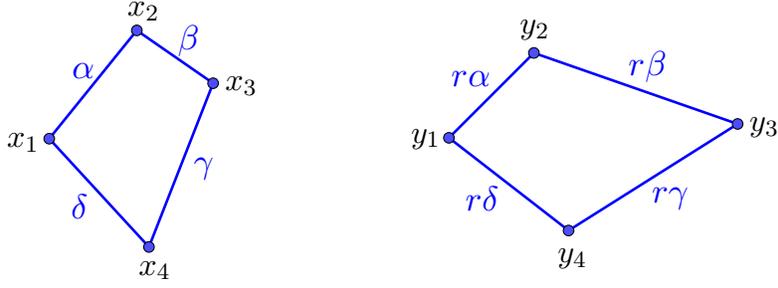
\begin{figure}[htbp]
\centering
\usetikzlibrary{arrows}

\definecolor{qqqqff}{rgb}{0.,0.,1.}

\definecolor{ududff}{rgb}{0.30196078431372547,0.30196078431372547,1.}

\begin{tikzpicture}[line cap=round,line join=round,>=triangle 45,x=1.0cm,y=1.0cm]


\draw [line width=1.pt,color=qqqqff] (-0.8899043193705356,4.739608067768051)-- (-2.0540131037084093,3.2988793742805833);
\draw [line width=1.pt,color=qqqqff] (-2.0540131037084093,3.2988793742805833)-- (-0.7285427056999393,1.8581506807931154);
\draw [line width=1.pt,color=qqqqff] (-0.7285427056999393,1.8581506807931154)-- (0.12436868084464156,4.036532465346167);
\draw [line width=1.pt,color=qqqqff] (0.12436868084464156,4.036532465346167)-- (-0.8899043193705356,4.739608067768051);
\draw [line width=1.pt,color=qqqqff] (4.388925613567546,4.439936499522657)-- (3.259394317873371,3.3104052038284832);
\draw [line width=1.pt,color=qqqqff] (3.259394317873371,3.3104052038284832)-- (4.849958795483535,2.07714144220321);
\draw [line width=1.pt,color=qqqqff] (4.849958795483535,2.07714144220321)-- (7.097495557323984,3.494818476594879);
\draw [line width=1.pt,color=qqqqff] (7.097495557323984,3.494818476594879)-- (4.388925613567546,4.439936499522657);


\draw [fill=ududff] (-0.8899043193705356,4.739608067768051) circle (2.0pt);
\draw[color=black] (-0.8,5) node {\scalebox{1.1}{$x_2$}};
\draw [fill=ududff] (-2.0540131037084093,3.2988793742805833) circle (2.0pt);
\draw[color=black] (-2.4,3.25) node {\scalebox{1.1}{$x_1$}};
\draw [fill=ududff] (-0.7285427056999393,1.8581506807931154) circle (2.0pt);
\draw[color=black] (-0.65,1.55) node {\scalebox{1.1}{$x_4$}};
\draw [fill=ududff] (0.12436868084464156,4.036532465346167) circle (2.0pt);
\draw[color=black] (0.5,4) node {\scalebox{1.1}{$x_3$}};
\draw[color=qqqqff] (-1.6,4.2) node {\scalebox{1.2}{$\alpha$}};
\draw[color=qqqqff] (-1.65,2.4) node {\scalebox{1.2}{$\delta$}};
\draw[color=qqqqff] (0,2.9) node {\scalebox{1.2}{$\gamma$}};
\draw[color=qqqqff] (-0.2,4.6) node {\scalebox{1.2}{$\beta$}};
\draw [fill=ududff] (4.388925613567546,4.439936499522657) circle (2.0pt);
\draw[color=black] (4.4,4.75) node {\scalebox{1.1}{$y_2$}};
\draw [fill=ududff] (3.259394317873371,3.3104052038284832) circle (2.0pt);
\draw[color=black] (2.95,3.3) node {\scalebox{1.1}{$y_1$}};
\draw [fill=ududff] (4.849958795483535,2.07714144220321) circle (2.0pt);
\draw[color=black] (4.9,1.7) node {\scalebox{1.1}{$y_4$}};
\draw [fill=ududff] (7.097495557323984,3.494818476594879) circle (2.0pt);
\draw[color=black] (7.45,3.45) node {\scalebox{1.1}{$y_3$}};
\draw[color=qqqqff] (3.55,4.099924527859615) node {\scalebox{1.2}{$r\alpha$}};
\draw[color=qqqqff] (3.7,2.5) node {\scalebox{1.2}{$r\delta$}};
\draw[color=qqqqff] (6.2,2.5) node {\scalebox{1.2}{$r\gamma$}};
\draw[color=qqqqff] (5.9,4.261286141530212) node {\scalebox{1.2}{$r\beta$}};
\end{tikzpicture}
\label{4-cycles}
\caption{Pair of 4-cycles with dilation ratio $r\in \mathbb{F}_p^{*}$. }
\end{figure}

The following theorem considers the case of 2-simplexes (3-cycles).

\begin{thm} 
\label{trianglesthm}

If $r\in (\mathbb{F}_p)^2\setminus \{0\}$,\ $p$ is an odd prime and $E\subset \mathbb{F}_p^2$ with $|E|\geq 3p$, then 
    \begin{equation}
        \label{eq1.2}
        \left|\left\{(x_1,x_2,x_3,y_1,y_2,y_3)\in E^6: \begin{array}{l}\lVert y_i-y_{i+1}\rVert=r\lVert 
        x_i-x_{i+1}\rVert,\ i\in [2],\\
        \lVert y_3-y_{1}\rVert=r\lVert x_3-x_{1}\rVert,\\
x_i\neq x_j,\ y_i\neq y_j,\ i\neq j \end{array}\right\}\right|>0\\
    \end{equation}
\end{thm}

\begin{figure}[htbp]
\centering
\usetikzlibrary{arrows}

\definecolor{qqqqff}{rgb}{0.,0.,1.}
\definecolor{ududff}{rgb}{0.30196078431372547,0.30196078431372547,1.}

\begin{tikzpicture}[line cap=round,line join=round,>=triangle 45,x=0.7cm,y=0.7cm]


\draw [line width=1.pt,color=qqqqff] (6.602090307158065,1.138309954574797)-- (11.143061728636795,0.20995705544345206);
\draw [line width=1.pt,color=qqqqff] (11.143061728636795,0.20995705544345206)-- (9.676553212222107,4.606730113868838);
\draw [line width=1.pt,color=qqqqff] (9.676553212222107,4.606730113868838)-- (6.602090307158065,1.138309954574797);
\draw [line width=1.pt,color=qqqqff] (-0.12919956920915274,0.4058961577577476)-- (3.270920147421271,0.4289478168535471);
\draw [line width=1.pt,color=qqqqff] (3.270920147421271,0.4289478168535471)-- (1.5508969667297192,3.3620120378159415);
\draw [line width=1.pt,color=qqqqff] (1.5508969667297192,3.3620120378159415)-- (-0.12919956920915274,0.4058961577577476);

\draw [fill=ududff] (6.602090307158065,1.138309954574797) circle (2.0pt);
\draw[color=black] (6.1,1.1) node {\scalebox{1.1}{$y_1$}};
\draw [fill=ududff] (11.143061728636795,0.20995705544345206) circle (2.0pt);
\draw[color=black] (11.7,0.15) node {\scalebox{1.1}{$y_3$}};
\draw[color=qqqqff] (8.8,0.2) node {\scalebox{1.2}{$r\gamma$}};
\draw[color=qqqqff] (10.9,2.5669891979889496) node {\scalebox{1.2}{$r\beta$}};
\draw[color=qqqqff] (7.7,3.1) node {\scalebox{1.2}{$r\alpha$}};
\draw [fill=ududff] (9.676553212222107,4.606730113868838) circle (2.0pt);
\draw[color=black] (9.75,5.05) node {\scalebox{1.1}{$y_2$}};
\draw [fill=ududff] (-0.12919956920915274,0.4058961577577476) circle (2.0pt);
\draw[color=black] (-0.6,0.35) node {\scalebox{1.1}{$x_1$}};
\draw [fill=ududff] (3.270920147421271,0.4289478168535471) circle (2.0pt);
\draw[color=black] (3.8,0.4) node {\scalebox{1.1}{$x_3$}};
\draw[color=qqqqff] (1.65,0) node {\scalebox{1.2}{$\gamma$}};
\draw[color=qqqqff] (2.7,2.1) node {\scalebox{1.2}{$\beta$}};
\draw[color=qqqqff] (0.4,2.1) node {\scalebox{1.2}{$\alpha$}};
\draw [fill=ududff] (1.5508969667297192,3.3620120378159415) circle (2.0pt);
\draw[color=black] (1.67,3.7) node {\scalebox{1.1}{$x_2$}};
\end{tikzpicture}
\label{3-cycles}
\caption{Pair of 2-simplexes (3-cycles) with dilation ratio $r\in (\mathbb{F}_p)^2\setminus \{0\}$. }
\end{figure}
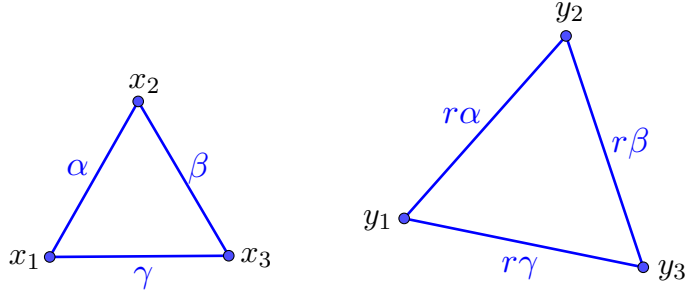

The following theorem generalizes the Theorem \ref{trianglesthm} to higher dimensions. For example, it tells us that if $E\subset \mathbb{F}_p^3$ such that $|E|\gg p^{\frac{3}{2}}$, then there exists a pair of 3-simplexes (tetrahedrons) in $E$ with dilation ratio $r\in (\mathbb{F}_p)^2\setminus \{0\}$.

\begin{figure}[htbp]
\centering
\usetikzlibrary{arrows}

\definecolor{qqqqff}{rgb}{0.,0.,1.}

\definecolor{ududff}{rgb}{0.30196078431372547,0.30196078431372547,1.}
\begin{tikzpicture}[scale=0.6][line cap=round,line join=round,>=triangle 45,x=1.0cm,y=1.0cm]
\clip(3.,-4.5) rectangle (26.,5.3);
\draw [line width=1.pt,dash pattern=on 3pt off 3pt,color=qqqqff] (4.92,-0.54)-- (13.86,-1.);
\draw [line width=1.pt,color=qqqqff] (13.86,-1.)-- (11.88,-3.42);
\draw [line width=1.pt,color=qqqqff] (11.88,-3.42)-- (4.92,-0.54);
\draw [line width=1.pt,color=qqqqff] (11.92,4.02)-- (4.92,-0.54);
\draw [line width=1.pt,color=qqqqff] (4.92,-0.54)-- (11.88,-3.42);
\draw [line width=1.pt,color=qqqqff] (11.88,-3.42)-- (13.86,-1.);
\draw [line width=1.pt,color=qqqqff] (13.86,-1.)-- (11.92,4.02);
\draw [line width=1.pt,color=qqqqff] (11.92,4.02)-- (11.88,-3.42);
\draw [line width=1.pt,color=qqqqff] (11.88,-3.42)-- (13.86,-1.);
\draw [line width=1.pt,color=qqqqff] (13.86,-1.)-- (11.92,4.02);

\draw [line width=1.pt,color=qqqqff] (18.28,-1.42)-- (21.06,-3.4);
\draw [line width=1.pt,color=qqqqff] (21.06,-3.4)-- (23.62,-1.72);
\draw [line width=1.pt,dash pattern=on 3pt off 3pt,color=qqqqff] (23.62,-1.72)-- (18.28,-1.42);
\draw [line width=1.pt,color=qqqqff] (21.04,3.36)-- (18.28,-1.42);
\draw [line width=1.pt,color=qqqqff] (18.28,-1.42)-- (21.06,-3.4);
\draw [line width=1.pt,color=qqqqff] (21.06,-3.4)-- (21.04,3.36);
\draw [line width=1.pt,color=qqqqff] (21.04,3.36)-- (23.62,-1.72);
\draw [line width=1.pt,color=qqqqff] (23.62,-1.72)-- (21.06,-3.4);
\draw [line width=1.pt,color=qqqqff] (21.06,-3.4)-- (21.04,3.36);

\draw [fill=ududff] (4.92,-0.54) circle (3.5pt);
\draw[color=black] (4.3,-0.5) node { \scalebox{1.2} {$x_1$} };
\draw [fill=ududff] (13.86,-1.) circle (3.5pt);
\draw[color=black] (14.5,-1) node { \scalebox{1.2} {$x_3$} };
\draw [fill=ududff] (11.88,-3.42) circle (3.5pt);
\draw[color=black] (11.8,-4) node { \scalebox{1.2} {$x_4$} };
\draw [fill=ududff] (11.92,4.02) circle (3.5pt);
\draw[color=black] (12,4.5) node { \scalebox{1.2} {$x_2$} };
\draw [fill=ududff] (18.28,-1.42) circle (3.5pt);
\draw[color=black] (17.7,-1.4) node { \scalebox{1.2} {$y_1$} };
\draw [fill=ududff] (21.06,-3.4) circle (3.5pt);
\draw[color=black] (21,-4) node { \scalebox{1.2} {$y_4$} };
\draw [fill=ududff] (23.62,-1.72) circle (3.5pt);
\draw[color=black] (24.3,-1.8) node { \scalebox{1.2} {$y_3$} };
\draw [fill=ududff] (21.04,3.36) circle (3.5pt);
\draw[color=black] (21.2,4) node { \scalebox{1.2} {$y_2$} };

\draw[color=qqqqff] (8,2.2) node {\scalebox{1.1}{$t_{12}$}};
\draw[color=qqqqff] (10,-0.3) node {\scalebox{1.1}{$t_{13}$}};
\draw[color=qqqqff] (9,-2.7) node {\scalebox{1.1}{$t_{14}$}};
\draw[color=qqqqff] (13.3,2.2) node {\scalebox{1.1}{$t_{23}$}};
\draw[color=qqqqff] (11.4,1.1) node {\scalebox{1.1}{$t_{24}$}};
\draw[color=qqqqff] (13.3,-2.7) node {\scalebox{1.1}{$t_{34}$}};

\draw[color=qqqqff] (19.2,1.7) node {\scalebox{1.1}{$rt_{12}$}};
\draw[color=qqqqff] (20.2,-1) node {\scalebox{1.1}{$rt_{13}$}};
\draw[color=qqqqff] (19.2,-2.7) node {\scalebox{1.1}{$rt_{14}$}};
\draw[color=qqqqff] (22.8,1.7) node {\scalebox{1.1}{$rt_{23}$}};
\draw[color=qqqqff] (21.8,0) node {\scalebox{1.1}{$rt_{24}$}};
\draw[color=qqqqff] (23,-2.9) node {\scalebox{1.1}{$rt_{34}$}};

\end{tikzpicture}

\caption{Pair of tetrahedrons with dilation ratio $r\in (\mathbb{F}_p)^2\setminus \{0\}$.}
\end{figure}
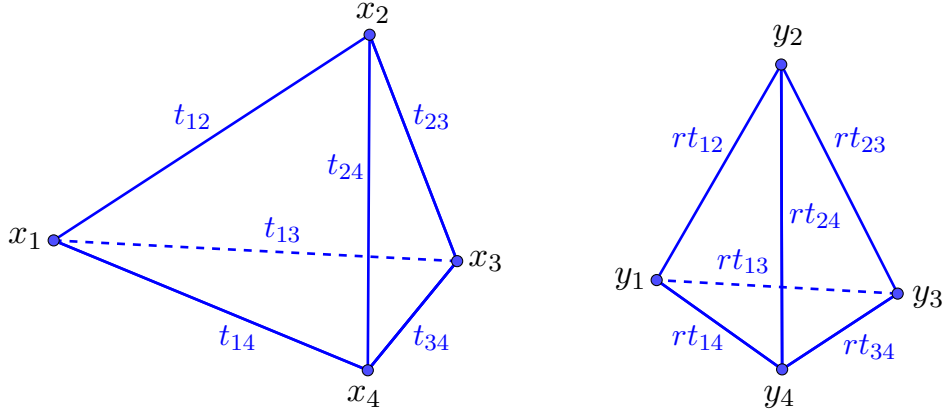

\begin{thm}
\label{ptsconfigthm}
If $r\in (\mathbb{F}_p)^2\setminus \{0\}$, $p$ is an odd prime and $E\subset \mathbb{F}_p^d$ such that $|E|\geq (d+1)p^{\frac{d}{2}}$, then 
    \begin{equation*}
            \left|\left\{(x_1,\dots,x_{d+1},y_1,\dots,y_{d+1})\in E^{2d+2}: \begin{array}{l}\lVert y_i-y_{j}\rVert=r\lVert x_i-x_{j}\rVert,\ i,j\in [d+1],\\
x_i\neq x_j,\ y_i\neq y_j,\ i\neq j \end{array}\right\}\right|>0.\\
    \end{equation*}
\end{thm}

\medskip

The following theorem gives the lower bound for the number of paths of length $k\geq 3$ in graphs in terms of edge density $e(G)/\binom{\nu(G)}{2}$. This result has been established in 1959 by Mulholland and Smith (see ~\cite{MR110721}).

\begin{thm}
\label{3path}
If $G=(V,E)$ is a simple graph with $n$ vertices, then
\begin{equation*}
    \left|\left\{(v_1,v_2,\dots,v_{k+1})\in V^{k+1}:v_iv_{i+1}\in E,\ i\in [k]\right\}\right|\geq \frac{(2e(G))^k}{n^{k-1}},
\end{equation*}
where $e(G)$ is the number of edges of graph $G$.
\end{thm}

Generalization of this inequality for bipartite graphs is called Sidorenko's conjecture. This conjecture still remains open but it has been proved for certain family of bipartite graphs such as paths, trees, cycles of even length, complete bipartite graphs, etc. 
Li and Szegedy in 2011  (see ~\cite{LiSz}) introduced the idea of using entropy to prove some cases of Sidorenko's conjecture. Later, Szegedy (see ~\cite{Sz}) applied these ideas to prove that an even wider class of bipartite graphs have Sidorenko's property. 

\smallskip

Using Theorem \ref{3path} we obtain the lower bound for the number of all pairs of $k$-paths ($k\geq 3$) in $E$ with dilation ratio $r\in \mathbb{F}_p^{*}$. We would like to notice that here by all we mean degenerate and nondegenerate pairs.

\begin{thm}
\label{lowerboundS_3(r)}
If $r\in \mathbb{F}_p^{*}$, $p$ is a prime such that $p \equiv 3 \pmod 4$ and $E\subset \mathbb{F}_p^2$ with $|E|>2p$, then
    \begin{equation*}
            \left|\left\{(x_1,\dots,x_{k+1},y_1,\dots,y_{k+1})\in E^{2k+2}: \begin{array}{l}\lVert y_i-y_{i+1}\rVert=r\lVert x_i-x_{i+1}\rVert,\\
x_i\neq x_{i+1},\ i\in [k] \end{array}\right\}\right|> \frac{|E|^{2k+2}}{(3p)^k}.\\
    \end{equation*}
\end{thm}

\smallskip

\begin{acknowledgement}
	This paper is a part of the author's Ph.D. thesis. The author would like to thank his advisors Kaave Hosseini and Alexander Iosevich for suggesting the interesting problem and for being a constant source of advice and encouragement.
\end{acknowledgement}  

\bigskip

\section{Preliminaries}
\label{sec:2}

We recall some definitions. Let $\mathbb{F}_p$ be the finite field of order $p$  and $\mathbb{F}_p^d$ be a $d$-dimensional vector space over $\mathbb{F}_p$. 

\smallskip

Let $\mathrm{O}_d(\mathbb{F}_p)$ denote the group of orthogonal $d\times d$ matrices with entries in $\mathbb{F}_p$.

\smallskip

Let $\mathrm{SO}_d(\mathbb{F}_p)$ denote the subgroup of the elements of $\mathrm{O}_d(\mathbb{F}_p)$ with determinant $1$.

\smallskip

Let $\mathbb{F}_p^{*}$ denote the set of nonzero elements in $\mathbb{F}_p$.

\smallskip

Let $(\mathbb{F}_p)^2$ denote the set of quadratic residues in $\mathbb{F}_p$.

\smallskip

Consider the map $\lVert \cdot \rVert:\mathbb{F}_p^d\to \mathbb{F}_p$ defined by $\lVert {\alpha}\rVert:=\alpha_1^2+\dots+\alpha_d^2$, where $\alpha=(\alpha_1,\dots,\alpha_d)$. 

This mapping is not a norm since we do not impose any metric structure on $\mathbb{F}_p^d$, but it does share the important feature of the Euclidean norm: it is invariant under orthogonal matrices.

\smallskip

If $t\in \mathbb{F}_p$, then let $S_t$ denote the sphere of radius $t$ in $\mathbb{F}_p^d$: thus $S_t=\{x\in \mathbb{F}_p^d: \lVert x\rVert=t\}.$

\smallskip

If $X$ is a finite set, then let $|X|$ denote the size (the cardinality) of $X$.

\smallskip

If $n\in \mathbb{N}$, then $[n]$ denote the set $\{1,\dots,n\}$.

\smallskip

For any two sets $A$ and $B$, let $A\sqcup B$ denote the disjoint union of $A$ and $B$.

\smallskip

We write $X\gg Y$ to mean that there is some constant $C>0$ so that $X\geq CY$.

\smallskip

We use $X\gg_d Y$ as shorthand for the inequality $X\geq C_dY$ for some constant $C_d>0$ depending only on $d$.

\smallskip

If $A\subset \mathbb{F}_p^d$, then let $\mathds{1}_A(x)$ to denote the indicator function of $A$: thus $$\mathds{1}_A(x)=
\begin{cases}
1, & \text{if }x\in A \\
0, & \text{if }x\notin A.
\end{cases}$$  

\smallskip

If $E\subset \mathbb{F}_p^2$ and $t_1,\dots,t_k\in \mathbb{F}_p$ with $k\geq 1$, then define the following incidence function: 
\begin{equation*}
    \nu_k(t_1,\dots,t_k):=\left|\left\{(x_1,\dots,x_{k+1})\in E^{k+1}: \lVert x_i-x_{i+1}\rVert=t_i,\ i\in [k]\right\}\right|.
\end{equation*}

\smallskip

Basically, this function counts the number of walks of length $k$ with step length $t_1,t_2,\dots,t_k$ in the complete graph $K_{|E|}$ with vertex set $E$.

\medskip

Likewise, if $E\subset \mathbb{F}_p^2$ and $t_1,t_2,t_3,t_4\in \mathbb{F}_p$, then define the following incidence function: 
\begin{equation*}
    \mu(t_1,t_2,t_3,t_4):=\left|\left\{(x_1,x_2,x_3,x_4)\in E^{4}: \begin{array}{l}\lVert x_1-x_{2}\rVert=t_1,\ \lVert x_2-x_{3}\rVert=t_2,\\
\lVert x_3-x_{4}\rVert=t_3,\ \lVert x_4-x_{1}\rVert=t_4 \end{array}\right\}\right|.\\
\end{equation*}

This function counts the number of closed walks of length 4 with step length $t_1,t_2,t_3,t_4$ in the complete graph $K_{|E|}$ with vertex set $E$.

\medskip

For $r\in \mathbb{F}_p^{*}$ and $k\geq 1$ we will consider the following sums:

\begin{equation*}
    \sum \limits_{t_1,\dots,t_k\in \mathbb{F}_p^*}\nu_{k}(t_1,\dots,t_k)\nu_{k}(rt_1,\dots,rt_k),
\end{equation*}
\begin{equation*}
    \sum \limits_{t_1,\dots,t_4\in \mathbb{F}_p^*}\mu(t_1,t_2,t_3,t_4)\mu(rt_1,rt_2,rt_3,rt_4).
\end{equation*}

\

Moreover, for each $k\geq 1$ we will consider the following sets:
\begin{equation*}
   S_k(r):=\left\{(x_1,\dots,x_{k+1},y_1,\dots,y_{k+1})\in E^{2k+2}: \lVert y_i-y_{i+1}\rVert=r\lVert x_i-x_{i+1}\rVert,\ x_i\neq x_{i+1},\ i\in [k] \right\},
\end{equation*}

\

\begin{equation*}
    C(r):=\left\{(x_1,x_2,x_3,x_4,y_1,y_2,y_3,y_4)\in E^8: \begin{array}{l}\lVert y_i-y_{i+1}\rVert=r\lVert x_i-x_{i+1}\rVert,\ x_i\neq x_{i+1},\ i\in [3], \\
\lVert y_4-y_{1}\rVert=r\lVert x_4-x_{1}\rVert,\ x_4\neq x_{1} \end{array}\right\}.\\
\end{equation*}

\

One can show that the above sums are equal to the size of the defined  sets, i.e.

\medskip

\begin{equation*}
    \sum \limits_{t_1,\dots,t_k\in \mathbb{F}_p^*}\nu_{k}(t_1,\dots,t_k)\nu_{k}(rt_1,\dots,rt_k)= |S_{k}(r)|,
\end{equation*}

\smallskip

\begin{equation*}
    \sum \limits_{t_1,t_2,t_3,t_4\in \mathbb{F}_p^*}\mu(t_1,t_2,t_3,t_4)\mu(rt_1,rt_2,rt_3,rt_4)= |C(r)|.
\end{equation*}

We need the estimates for the size of the sets $S_1(r)$, $S_{2}(r)$ and $C(r)$, which are given in Lemma \ref{lower bound for S(r)}, \ref{lower bound for S_2(r)} and \ref{lem2.4}. First of all, we need to prove the following 

\begin{lem}
\label{SO_2 lemma}
Let $p$ be a prime number such that $p\equiv 3 \pmod 4$. If $u,v\in \mathbb{F}_p^2\setminus \{(0,0)\}$ and $r\in (\mathbb{F}_p)^2\setminus\{0\}$ such that $\lVert u\rVert=r\lVert v\rVert$, then there exists unique $\theta \in \mathrm{SO}_2(\mathbb{F}_p)$ such that $u=\sqrt{r}\theta v,$ where $\sqrt{r}$ is an element of $\mathbb{F}_p$ whose square is $r$.
\end{lem}

\begin{proof}
    Suppose that $u=\begin{pmatrix}
           u_{1} \\
           u_{2}
         \end{pmatrix}$, $v=\begin{pmatrix}
           v_{1} \\
           v_{2}
         \end{pmatrix}$ and consider matrices $U=\begin{pmatrix}
           u_{1} & -u_2 \\
           u_{2} & u_1
         \end{pmatrix}$, $V=\begin{pmatrix}
           v_{1} & -v_2 \\
           v_{2} & v_1
         \end{pmatrix}$. It's clear that $U^{\top}U=\lVert u\rVert\mathbf{I}_2$, $V^{\top}V=\lVert v\rVert\mathbf{I}_2$, where $\mathbf{I}_2$ is an identity $2\times 2$ matrix. Since $U\begin{pmatrix}
           1 \\
           0
         \end{pmatrix}=u$ and $V\begin{pmatrix}
           1 \\
           0
         \end{pmatrix}=v$, hence $UV^{-1}v=u$. Thus $u=\sqrt{r}\theta v$, where $\theta=\frac{1}{\sqrt{r}}UV^{-1}$. It is not difficult to check that $\theta\in \mathrm{SO}_2(\mathbb{F}_p)$ and it is unique since the matrix $h\in \mathrm{SO}_2(\mathbb{F}_p)$ such that $h\begin{pmatrix}
           1 \\
           0
         \end{pmatrix}=\begin{pmatrix}
           1 \\
           0
         \end{pmatrix}$ is unique.
\end{proof}

\medskip

\begin{lem}
\label{lower bound for S(r)}
The following inequality holds:
\begin{equation*}
    |S_1(r)|\geq \left(\frac{1}{p}+\frac{1}{p^2}-\frac{1}{p^3}\right)|E|^4-\frac{2|E|^3}{p}-(p+1)|E|^2.
\end{equation*}    
\end{lem}
\begin{proof}
    This inequality follows from (2.7) in \cite{MR3959878} by taking $d=2$. 
\end{proof}

\medskip

\begin{lem}
\label{lower bound for S_2(r)}
The following inequality holds:
\begin{equation}
\label{inequality for l.b.S_2(r)}
    |S_2(r)|\geq |E|^{-2}|S_1(r)|^2.
\end{equation}
\end{lem}

\medskip

\begin{remark}
    We shall prove the Lemma \ref{lower bound for S_2(r)} in two ways. The first proof is algebraic and relies on the group actions approach. However, the second proof is done by means of graph theory. It is worth noting that graph-theoretic proof gives a better result since it establishes the inequality \eqref{inequality for l.b.S_2(r)} for all $r\in \mathbb{F}_p^{*}$, but algebraic method proves it only for $r\in (\mathbb{F}_p)^2\setminus \{0\}$.      
\end{remark}

\begin{proof}[The first proof of Lemma \ref{lower bound for S_2(r)}] 

Assume that $r\in (\mathbb{F}_p)^2\setminus \{0\}$ and if we apply Lemma \ref{SO_2 lemma}, then we obtain:

\begin{equation*}
\begin{split}
|S_{2}(r)|&=\left|\left\{(x_1,x_2,x_3,y_1,y_2,y_3)\in E^6: \begin{array}{l}\lVert y_i-y_{i+1}\rVert=r\lVert x_i-x_{i+1}\rVert,\\
x_i\neq x_{i+1},\ i\in [2] \end{array}\right\}\right|\\[5pt]
&=\sum_{\theta_1,\theta_2}\left|\left\{(x_1,x_2,x_3,y_1,y_2,y_3)\in E^6: \begin{array}{l} y_i-y_{i+1}=\sqrt{r}\theta_i( x_i-x_{i+1}),\\
x_i\neq x_{i+1},\ i\in [2] \end{array}\right\}\right|\\[5pt]
&=\sum_{\theta_1,\theta_2}\left|\left\{(x_1,x_2,x_3,y_1,y_2,y_3)\in E^6: \begin{array}{l} y_i-\sqrt{r}\theta_ix_i=y_{i+1}-\sqrt{r}\theta_ix_{i+1},\\
x_i\neq x_{i+1},\ i\in [2] \end{array}\right\}\right|\\[5pt]
&=\sum_{\substack{\theta_1,\theta_2\\a,b}}\left|\left\{(x_1,x_2,x_3,y_1,y_2,y_3)\in E^6: \begin{array}{l}y_1-\sqrt{r}\theta_1 x_1=y_{2}-\sqrt{r}\theta_1 x_{2}=a,\\
y_2-\sqrt{r}\theta_2 x_2=y_{3}-\sqrt{r}\theta_2 x_{3}=b, \\ x_1\neq x_{2},\ x_2\neq x_3 \end{array}\right\}\right|\\[5pt]
&=\sum_{\substack{\theta_1,\theta_2\\a,b}}\,
  \smashoperator[r]{%
  \sum_{\substack{y_2-\sqrt{r}\theta_1x_2=a\\ y_2-\sqrt{r}\theta_2x_2=b}}}
  \oneE(x_2)\oneE(y_2)
  \smashoperator[r]{%
  \sum_{\substack{y_1-\sqrt{r}\theta_1x_1=a\\ x_1\neq x_2}}}
  \oneE(x_1)\oneE(y_1) 
  \smashoperator[r]{%
  \sum_{\substack{y_3-\sqrt{r}\theta_2x_3=b\\ x_3\neq x_2}}}
  \oneE(x_3)\oneE(y_3) \\[5pt]
\end{split}
\end{equation*}
\begin{equation*}
\begin{split}
&=\sum _{x_2,y_2\in E}
  \begin{aligned}[t]
  &\underbrace{\Biggl(\,
  \sum_{\theta_1,a}\,
  \smashoperator[r]{%
  \sum _{\substack{y_1-\sqrt{r}\theta_1x_1=a\\ x_1\neq x_2}}}
  \oneE(x_1)\oneE(y_1)\oneE[\{a\}](y_2-\sqrt{r}\theta_1x_2)
  \!\Biggr)}_{\mathrm{I}}  \\
  &\times\underbrace{\Biggl(\,
  \sum_{\theta_2,b}\,
  \smashoperator[r]{%
  \sum _{\substack{y_3-\sqrt{r}\theta_2x_3=b\\ x_2\neq x_3}}}
  \oneE(x_3)\oneE(y_3)\oneE[\{b\}](y_2-\sqrt{r}\theta_2x_2)
  \!\Biggr).}_{\mathrm{II}} 
  \end{aligned}
\end{split}  
\end{equation*}

\smallskip

One can easily check that $\mathrm{I}$ and $\mathrm{II}$ can be written as follows:

\medskip

\begin{equation*}
\begin{split}
    \mathrm{I}&=\sum \limits_{\theta_1}\sum \limits_{\substack{y_1-\sqrt{r}\theta_1x_1=y_2-\sqrt{r}\theta_1x_2\\ x_1\neq x_2}}\mathds{1}_{E}(x_1)\mathds{1}_{E}(y_1)\\
    &=\sum \limits_{\theta_1}\left|\left\{(x_1,y_1)\in E^2: y_1-\sqrt{r}\theta_1x_1=y_2-\sqrt{r}\theta_1x_2,\ 
x_2\neq x_{3}\right\}\right|,
\end{split}
\end{equation*}

\begin{equation*}
    \mathrm{II}=\sum \limits_{\theta_2}\left|\left\{(x_3,y_3)\in E^2: y_3-\sqrt{r}\theta_2x_3=y_2-\sqrt{r}\theta_2x_2,\ 
x_1\neq x_{2}\right\}\right|.
\end{equation*}

If we let $\lambda(r,\theta,x_2,y_2)\coloneqq\left|\left\{(x,y)\in E^2: y-\sqrt{r}\theta x=y_2-\sqrt{r}\theta x_2,\
x\neq x_{2}\right\}\right|
$, then we obtain

\begin{equation*}
    |S_{2}(r)|=\sum \limits_{x_2,y_2\in E}\left(\sum \limits_{\theta}\lambda (r,\theta,x_2,y_2) \right)^2.
\end{equation*}

Applying Cauchy-Schwarz inequality, we obtain
\begin{equation*}
    \left( \sum \limits_{x_2,y_2\in E}\sum \limits_{\theta}\lambda (r,\theta,x_2,y_2)\right)^2\leq |E|^2\sum \limits_{x_2,y_2\in E}\left(\sum \limits_{\theta}\lambda (r,\theta,x_2,y_2) \right)^2.
\end{equation*}
Hence 
\begin{equation*}
    |S_{2}(r)|\geq |E|^{-2} \left( \sum \limits_{x_2,y_2\in E}\sum \limits_{\theta}\lambda (r,\theta,x_2,y_2)\right)^2.
\end{equation*}

We shall show that the inner sum is $|S_1(r)|$. Indeed,

\begin{align*}
     &\sum \limits_{x_2,y_2\in E}\sum \limits_{\theta\in \mathrm{SO}_2(\mathbb{F}_p)}\lambda (r,\theta,x_2,y_2)\\[5pt]
     &=\sum \limits_{x_2,y_2\in E}\sum \limits_{\theta\in \mathrm{SO}_2(\mathbb{F}_p)}\left|\left\{(x,y)\in E^2: y-\sqrt{r}\theta x=y_2-\sqrt{r}\theta x_2,\
x\neq x_{2}\right\}\right|\\[5pt]
     &=\sum \limits_{x_2,y_2\in E}\left|\left\{(x,y)\in E^2: \lVert y-y_2 \rVert=r\lVert x-x_2 \rVert,\
x\neq x_{2}\right\}\right|\\[5pt]
     &=\left|\left\{(x,x_2,y,y_2)\in E^4: \lVert y-y_2 \rVert=r\lVert x-x_2 \rVert,\
x\neq x_{2}\right\}\right|= |S_1(r)|,
\end{align*} which completes the proof.
\end{proof}

\smallskip

\begin{proof}[The second proof of Lemma \ref{lower bound for S_2(r)}]
            Consider the graph $G=(V,E)$ defined as follows: let's introduce the following notations: \begin{align*}
        L:=E\times E=\{(x,y): x,y\in E\}, 
    \end{align*}
    \begin{align*}
        R:=S_1(r)=\{(x_1,x_2,y_1,y_2)\in E^4: \lVert y_1-y_2\rVert=r\lVert x_1-x_2\rVert,\ x_1\neq x_2\}.
    \end{align*} 
    
    Define the vertex set $V$ to be $L\cup R$ and the edge set $E$ we will define in the following way: each vertex $(x_1,x_2,y_1,y_2)\in R$ we join with $(x_1,y_1)\in L$ and $(x_2,y_2)\in L$ and we notice that $(x_1,y_1)\neq (x_2,y_2)$. 

    \smallskip

    For any vertex $v\in V$, let $d_v$ denote it's degree, i.e. \begin{equation*}
        d_v:=|\{e\in E: e \ \text{is incident with} \ v\}|.
    \end{equation*} 

    From degree sum formula, we obtain \begin{equation*}
        \sum \limits_{v\in L} d_v+\sum \limits_{v\in R} d_v=2|E|.
    \end{equation*} 
    
    If $v\in R$, then $d_v=2$ and $|E|=2|R|=2|S_1(r)|$ and hence 
    \begin{equation*}
        \sum \limits_{v\in L} d_v=2|S_1(r)|.
    \end{equation*}

    Applying Cauchy-Schwarz, we obtain \begin{equation}
    \label{eq2.1}
        \sum \limits_{v\in L} d_v^2\geq |L|^{-1}\left( \sum \limits_{v\in L} d_v\right)^2=4|E|^{-2}|S_1(r)|^2.
    \end{equation}

    We see that $d_v=|\{u\in R: uv\in E\}|$ because the graph $G$ is simple (without loops and multiple edges), then \begin{equation}
    \label{eq2.2}
        \sum \limits_{v\in L} d_v^2=|\{(u,u',v)\in R\times R\times L: uv\in E,\ u'v\in E\}|.  
    \end{equation}

    We will show that the quantity on the RHS of \eqref{eq2.2} is equal to $4|S_{2}(r)|$. 

\

    Indeed, let's consider a function \begin{equation*}
        f: \{(u,u',v)\in R\times R\times L: uv\in E,\ u'v\in E\}\to S_{2}(r)
    \end{equation*} defined as follows: each element of the domain is $(u,u',v)\in R\times R\times L$ with $uv\in E$, $u'v\in E$ and hence $u=(a,b,a',b')\in S_1(r),\ u'=(c,d,c',d')\in S_1(r)$. The vertex $v\in L$ being incident with both $u=(a,b,a',b')$ and $u'=(c,d,c',d')$ implies that $$((v=(a,a'))\lor (v=(b,b')))\land ((v=(c,c'))\lor (v=(d,d'))).$$ 
    
    Therefore, we have the following 4 cases:
    \begin{itemize}
        \item if $v=(a,a')=(c,c'),$ then \begin{equation*}
            (u,u',v)\equiv((a,b,a',b'),(a,d,a',d'),(a,a'))\xmapsto{f} (b,a,d,b',a',d').
        \end{equation*}
    \end{itemize}
    
    \begin{itemize}
        \item if $v=(a,a')=(d,d'),$ then \begin{align*}
            (u,u',v)\equiv((a,b,a',b'),(c,a,c',a'),(a,a'))\xmapsto{f} (b,a,c,b',a',c').
        \end{align*}
    \end{itemize}
    
    \begin{itemize}
        \item if $v=(b,b')=(c,c'),$ then \begin{equation*}
            (u,u',v)=((a,b,a',b'),(b,d,b',d'),(b,b'))\xmapsto{f} (a,b,d,a',b',d').
        \end{equation*}
    \end{itemize}
    \begin{itemize}
        \item if $v=(b,b')=(d,d'),$ then \begin{equation*}
            (u,u',v)=((a,b,a',b'),(c,b,c',b'),(b,b'))\xmapsto{f} (a,b,c,a',b',c').
        \end{equation*}
    \end{itemize}

    Defining $f$ in this way implies that for each $\alpha\in S_{2}(r)$ the preimage of $\alpha$ under $f$ has exactly $4$ elements: thus $|f^{-1}(\{\alpha\})|=4.$ It immediately implies that \begin{equation}
    \label{eq2.3}
        |\{(u,u',v)\in R\times R\times L: uv\in E,\ u'v\in E\}|=4|S_{2}(r)|.
    \end{equation} 
    
    Comparing \eqref{eq2.3}  with \eqref{eq2.2} and \eqref{eq2.1}, we obtain the desired inequality 
    \begin{equation*}
        |S_{2}(r)|\geq |E|^{-2}|S_1(r)|^2. \qedhere
    \end{equation*}
    \end{proof}

\smallskip

\begin{lem}
\label{lem2.4}
    The following inequality holds:
    \begin{equation*}
        |C(r)|\geq |E|^{-4}|S_{2}(r)|^2.
    \end{equation*}
\end{lem}
    \begin{proof}
            We will consider the graph $G=(V,E)$ which is defined in the following way: let's introduce the following notations: \begin{equation*}
        L:= E\times E\times E\times E,
    \end{equation*}
    \begin{equation*}
        R:=S_{2}(r).
    \end{equation*} 
    
    So we define the vertex set $V$ of graph $G$ to be $L\cup R$. We will define edge set $E$ of the graph $G$ as follows: each vertex $v\in R$ which has form $(x_1,x_2,x_3,y_1,y_2,y_3)$ we connect by an edge with $(x_1,x_3,y_1,y_3)\in L$.

    From degree sum formula, we obtain \begin{equation*}
        \sum \limits_{v\in L}d_v+\sum \limits_{v\in R}d_v=2|E|,
    \end{equation*} where $d_v:=\deg(v)$. If $v\in R$, then $d_v=1$  and $|E|=|S_{2}(r)|$. Hence we have $$\sum \limits_{v\in L}d_v=|S_{2}(r)|.$$ 
    
    Applying Cauchy-Schwarz inequality, we obtain 
    \begin{equation}
    \label{eq2.4}
        \sum \limits_{v\in L}d_v^2\geq |E|^{-4}|S_{2}(r)|^2.
    \end{equation} 
    
    However,
    
\begin{equation} 
\label{eq2.5}
        \sum \limits_{v\in L}d_v^2  = |\{(u,u',v)\in R\times R\times L: uv\in E,\ u'v\in E\}|= |C(r)|.
\end{equation}

Comparing \eqref{eq2.5} with \eqref{eq2.4} we obtain the desired inequality 
\begin{equation*}
    |C(r)|\geq |E|^{-4}|S_{2}(r)|^2. \qedhere
\end{equation*}
    \end{proof}

\medskip

We also need to know the size of the sphere $S_t$ in $\mathbb{F}_q^d$ which is given in the following 

\begin{lem}
\label{lem2.5}
    Let $S_t$ denote the sphere of radius $t\in \mathbb{F}_q$ in $\mathbb{F}_q^d$. If $d\geq 2$ is even, then 
    \begin{equation*}
        |S_t|=q^{d-1}+\lambda(t)q^{\frac{d-2}{2}}\eta\left((-1)^{\frac{d}{2}}\right),
    \end{equation*}
    where $\eta$ is the quadratic character of $\mathbb{F}_q^*$, $\lambda(t)=-1$ for $t\in \mathbb{F}_q^*$, and $\lambda(0)=q-1$.
\end{lem}
\begin{proof}
    This follows from Theorem 6.26 in ~\cite{MR1429394}.
\end{proof}

\medskip

The trivial inequality between counting functions $\nu_{2}(\cdot,\cdot)$ and $\nu_1(\cdot)$ which comes in handy is given in the following

\begin{lem}
\label{lem2.6}
    If $t_1,t_2\in \mathbb{F}_p$, then the following inequality holds:
    \begin{equation*}
        \nu_{2}(t_1,t_2)\leq |E|\nu_1(t_1).
    \end{equation*}
\end{lem}
\begin{proof}
        By definition of the counting function $\nu_{2}(\cdot,\cdot)$, we obtain 
        \[\begin{split}
        \nu_{2}(t_1,t_2)&=|\{(x_1,x_2,x_3)\in E^3:\lVert x_1-x_2 \rVert=t_1,\ \lVert x_2-x_3 \rVert=t_2\}|\\[5pt]
        &=\sum\limits_{x_1,x_2,x_3\in \mathbb{F}_p^2}\mathds{1}_{E}(x_1)\mathds{1}_{E}(x_2)\mathds{1}_{E}(x_3)\mathds{1}_{S_{t_1}}(x_1-x_2)\mathds{1}_{S_{t_2}}(x_2-x_3) \\[5pt]
        &\leq\sum\limits_{x_1,x_2,x_3\in \mathbb{F}_p^2}\mathds{1}_{E}(x_1)\mathds{1}_{E}(x_2)\mathds{1}_{E}(x_3)\mathds{1}_{S_{t_1}}(x_1-x_2)\\[5pt]
        &=\sum \limits_{x_3\in \mathbb{F}_p^2}\mathds{1}_{E}(x_3)\sum\limits_{x_1,x_2\in \mathbb{F}_p^2}\mathds{1}_{E}(x_1)\mathds{1}_{E}(x_2)\mathds{1}_{S_{t_1}}(x_1-x_2)\\[5pt]
        &=|E|\nu_1(t_1). \qedhere
        \end{split}\]
\end{proof}

\medskip

The following lemma gives us the order of the group of orthogonal matrices over finite field.

\begin{lem}
\label{lem2.9}
If $F$ is any field and $\mathrm{O}_n(F)$ is a group of orthogonal $n\times n$ matrices with entries in $F$, then for any odd prime $p$ we have: $$|\mathrm{O}_{2n+1}(\mathbb{F}_p)|=2p^{n^2}\prod \limits_{i=1}^n(p^{2i}-1),$$ $$|\mathrm{O}_{2n}^+(\mathbb{F}_p)|=2p^{n(n-1)}(p^n-1)\prod \limits_{i=1}^{n-1}(p^{2i}-1),$$ $$|\mathrm{O}_{2n}^-(\mathbb{F}_p)|=2p^{n(n-1)}(p^n+1)\prod \limits_{i=1}^{n-1}(p^{2i}-1).$$
\end{lem}

\begin{proof}
This statement can be found on page 141 in ~\cite{MR1189139}.
\end{proof}

\bigskip

\section{Proof of Theorem \ref{2chainsthm}.}
\label{sec:3}

From the definition of a set $S_{2}(r)$, it is clear that it also contains pairs of $2$-paths in $E$ with dilation ratio $r$ which are degenerate. For instance, degenerate $2$-paths come if one takes $x_1=x_3$ or $y_1=y_3$. That is why we need to rule out all degenerate cases and in order to implement it we need to consider the following sets: 

\begin{equation*}
\mathcal{A}:= \left\{(x_1,x_2,x_3,y_1,y_2,y_3)\in E^6: \begin{array}{l}\lVert y_i-y_{i+1}\rVert=r\lVert x_i-x_{i+1}\rVert,\ i\in [2],\\
x_1\neq x_{2},\ x_2\neq x_3,\ x_1=x_3\end{array}\right\},
\end{equation*}

\begin{equation*}
\mathcal{B}:= \left\{(x_1,x_2,x_3,y_1,y_2,y_3)\in E^6: \begin{array}{l}\lVert y_i-y_{i+1}\rVert=r\lVert x_i-x_{i+1}\rVert,\ i\in [2],\\
x_1\neq x_{2},\ x_2\neq x_3,\ y_1=y_3\end{array}\right\},
\end{equation*}

\begin{equation*}
\mathcal{C}:= \left\{(x_1,x_2,x_3,y_1,y_2,y_3)\in E^6: \begin{array}{l}\lVert y_i-y_{i+1}\rVert=r\lVert x_i-x_{i+1}\rVert,\ i\in [2],\\
x_1\neq x_{2},\ x_2\neq x_3,\ x_1\neq x_3,\ y_1\neq y_3\end{array}\right\}.
\end{equation*}

It is easy to see that \begin{equation}
\label{eq3.1}
    \mathcal{A}\cup \mathcal{B}\cup\mathcal{C}=S_{2}(r).
\end{equation}

From the definition of sets $\mathcal{A},\mathcal{B}, \mathcal{C}$ follows that \begin{equation}
\label{eq3.2}
    \mathcal{A}\cap \mathcal{C}= \mathcal{B}\cap \mathcal{C}= \mathcal{A}\cap \mathcal{B}\cap \mathcal{C}=\varnothing.
\end{equation} 

One can see that the pairs of 2-paths in $E$ with dilation ratio $r$ is exactly the set $\mathcal{C}$ since $y_1\neq y_2$ follows from the fact that $\lVert y_1-y_2\rVert=r\lVert x_1-x_2\rVert$, $x_1\neq x_2$ and $\lVert x \rVert=0$ iff $x=(0,0)$ since $p\equiv 3 \pmod 4$. The same reasoning holds for $y_2\neq y_3$.

\smallskip

Applying  inclusion–exclusion principle to \eqref{eq3.1} and taking into account \eqref{eq3.2},  we obtain  \begin{equation}
\label{eq3.3}
    |\mathcal{C}|=|S_{2}(r)|-|\mathcal{A}|-|\mathcal{B}|+|\mathcal{A}\cap \mathcal{B}|.
\end{equation}

\medskip

Now we can explicitly compute the size of $|\mathcal{A}|, \ |\mathcal{B}|$ and $|\mathcal{A}\cap \mathcal{B}|$. Indeed, 

\medskip

\begin{equation*}
\begin{split}
     |\mathcal{A}\cap \mathcal{B}|&=\left|\left\{(x_1,x_2,x_3,y_1,y_2,y_3)\in E^6: \begin{array}{l}\lVert y_i-y_{i+1}\rVert=r\lVert x_i-x_{i+1}\rVert,\ i\in [2],\\
    x_1\neq x_{2},\ x_2\neq x_3,\ x_1=x_3,\ y_1=y_3 \end{array}\right\}\right|\\[5pt]
     &=\left|\left\{(x_1,x_2,y_1,y_2)\in E^4: \lVert y_1-y_{2}\rVert=r\lVert x_1-x_{2}\rVert,\
    x_1\neq x_{2} \right\}\right|\\[5pt]
    &=|S_1(r)|.
\end{split}
\end{equation*}

\medskip

Now we proceed to the size of $\mathcal{A}$:

\medskip

\begin{equation*}
\begin{split}
    |\mathcal{A}|&= \left|\left\{(x_1,x_2,x_3,y_1,y_2,y_3)\in E^6: \begin{array}{l}\lVert y_i-y_{i+1}\rVert=r\lVert x_i-x_{i+1}\rVert,\ i\in [2],\\
x_1\neq x_{2},\ x_2\neq x_3,\ x_1=x_3\end{array}\right\}\right|\\[5pt]
&=\left|\left\{(x_1,x_2,x_1,y_1,y_2,y_3)\in E^6: \begin{array}{l}\lVert y_1-y_{2}\rVert=\lVert y_2-y_{3}\rVert=r\lVert x_1-x_{2}\rVert,\\
x_1\neq x_{2}\end{array}\right\}\right|\\[5pt]
&=\sum \limits_{t\in \mathbb{F}_p^{*}}\left|\left\{(x_1,x_2,y_1,y_2,y_3)\in E^5: \begin{array}{l}\lVert y_1-y_{2}\rVert=\lVert y_2-y_{3}\rVert=rt,\\
\lVert x_1-x_{2}\rVert=t\end{array}\right\}\right|\\[5pt]
&=\sum \limits_{t\in \mathbb{F}_p^*} \sum \limits_{\substack{x_1,x_2,\\ y_1,y_2,y_3}}\mathds{1}_E(x_1)\mathds{1}_E(x_2)\mathds{1}_E(y_1)\mathds{1}_E(y_2)\mathds{1}_E(y_3)\mathds{1}_{S_t}(x_1-x_2)\mathds{1}_{S_{rt}}(y_1-y_2)\mathds{1}_{S_{rt}}(y_2-y_3)\\[5pt]       
&=\sum \limits_{t\in \mathbb{F}_p^*} \sum \limits_{x_1,x_2}\mathds1_E(x_1)\mathds1_E(x_2)\mathds1_{S_t}(x_1-x_2)\sum \limits_{y_1,y_2,y_3}\mathds1_E(y_1)\mathds1_E(y_2)\mathds1_E(y_3)\mathds1_{S_{rt}}(y_1-y_2)\mathds1_{S_{rt}}(y_2-y_3)\\[5pt]    
&=\sum \limits_{t\in \mathbb{F}_p^*} \nu_1(t) \nu_{2}(rt,rt).\\
\end{split}
\end{equation*}

\smallskip

In an analogous way one can compute the size of $\mathcal{B}$ and we obtain  the following equalities: 
\begin{equation}
\label{eq3.4}
\begin{split}
|\mathcal{A}\cap \mathcal{B}| = &|S_1(r)|,\quad |\mathcal{A}|=\sum \limits_{t\in \mathbb{F}_p^*} \nu_1(t) \nu_{2}(rt,rt), \\[5pt]
 & |\mathcal{B}|=\sum \limits_{t\in \mathbb{F}_p^*} \nu_1(rt) \nu_{2}(t,t),
\end{split}
\end{equation}

If we plug \eqref{eq3.4} into \eqref{eq3.3}, then we obtain
\begin{equation}
\label{eq3.5}
    |\mathcal{C}|=|S_{2}(r)|+|S_1(r)|-\sum \limits_{t\in \mathbb{F}_p^*} \nu_1(t) \nu_{2}(rt,rt)-\sum \limits_{t\in \mathbb{F}_p^*} \nu_1(rt) \nu_{2}(t,t).
\end{equation}

The upper bounds for the third and fourth terms in \eqref{eq3.5} can be obtained from Lemma \ref{lem2.6}. Indeed, since $\nu_{2}(rt,rt)\leq |E|\nu_1(rt),$ then 
\begin{align}
\label{eq3.6}
    &\sum \limits_{t\in \mathbb{F}_p^*} \nu_1(t) \nu_{2}(rt,rt)\leq |E|\sum \limits_{t\in \mathbb{F}_p^{*}}\nu_1(t)\nu_1(rt)=|E||S_1(r)|.
\end{align}

In an analogous way, we obtain \begin{align}\label{eq3.7}
    &\sum \limits_{t\in \mathbb{F}_p^*} \nu_1(rt) \nu_{2}(t,t)\leq |E|\sum \limits_{t\in \mathbb{F}_p^{*}}\nu_1(t)\nu_1(rt)=|E||S_1(r)|.
\end{align}

Using estimates \eqref{eq3.6} and \eqref{eq3.7} in  \eqref{eq3.5}, we obtain 
\begin{align*}
    |\mathcal{C}| \geq |S_{2}(r)|+|S_1(r)|-2|E||S_1(r)|.
\end{align*}

Applying Lemma \ref{lower bound for S_2(r)} we obtain that
\begin{equation} 
\label{eq3.8}
\begin{split}
    |\mathcal{C}| & \geq |E|^{-2}|S_1(r)|^2+|S_1(r)|-2|E||S_1(r)| \\[5pt]
 & = |E|^{-2}|S_1(r)|\left(|S_1(r)|+|E|^2-2|E|^3 \right).
\end{split}
\end{equation}

The following lemma shows the positivity of \eqref{eq3.8} which implies the proof of Theorem \ref{2chainsthm}.

\begin{lem}
\label{lemm3.1}
If $|E|>(\sqrt{3}+1)p$, then 
\begin{align*}
    |E|^{-2}|S_1(r)|>0 \quad  \text{and} \quad |S_1(r)|+|E|^2-2|E|^3>0.
\end{align*}
\end{lem} 
\begin{proof}
Let's start with the first inequality. Applying Lemma \ref{lower bound for S(r)}, we obtain 

\begin{equation*}
\begin{split}
    |E|^{-2}|S_1(r)|& \geq \left(\frac{1}{p}+\frac{1}{p^2}-\frac{1}{p^3}\right)|E|^2-\frac{2|E|}{p}-(p+1)\\
    &=\frac{1}{p^3}\bigg((p^2+p-1)|E|^2-2p^2|E|-p^3(p+1) \bigg)\\
    &=\frac{p^2+p-1}{p^3}\bigg(|E|-\frac{p^2-\sqrt{p^6+2p^5+p^4-p^3}}{p^2+p-1} \bigg)\\
    &\times\bigg(|E|-\frac{p^2+\sqrt{p^6+2p^5+p^4-p^3}}{p^2+p-1} \bigg).
\end{split}   
\end{equation*}
\medskip

It suffices to show that the expression in the second parenthesis is positive since it immediately implies that the expression in the first parenthesis is also positive.

\medskip

If $|E|>(\sqrt{3}+1)p$, then 

\begin{equation*}
\begin{split}
    |E|&-\frac{p^2+\sqrt{p^6+2p^5+p^4-p^3}}{p^2+p-1}\\
    &>p\bigg((\sqrt{3}+1)-\frac{p^2+\sqrt{p^6+2p^5+p^4-p^3}}{p^3+p^2-p}\bigg)\\
    &>p\bigg((\sqrt{3}+1)-\frac{p^2+\sqrt{p^6+2p^5+p^4}}{p^3+p^2-p}\bigg)\\
    &=p\bigg((\sqrt{3}+1)-\frac{p^3+2p^2}{p^3+p^2-p}\bigg).
\end{split}    
\end{equation*}

However, the function $\phi:[3,+\infty)\to \mathbb{R}$ defined by $\phi(x)=\mfrac{x^3+2x^2}{x^3+x^2-x}$ is decreasing and hence $\phi(x)\leq \phi(3)=\frac{15}{11}.$

\smallskip

Therefore, we obtain
\begin{equation*}
        |E|-\frac{p^2+\sqrt{p^6+2p^5+p^4-p^3}}{p^2+p-1}>p\bigg((\sqrt{3}+1)-\frac{15}{11} \bigg)>0.
\end{equation*}

We have shown that \begin{equation*}
    |E|^{-2}|S_1(r)|>0.
\end{equation*}

\medskip

Now we proceed to the second inequality. Let's use the Lemma \ref{lower bound for S(r)} and we obtain 
\begin{equation*}
\begin{split}
    |S_1(r)|&+|E|^2-2|E|^3\\
    &\geq\frac{|E|^2}{p^3}\bigg((p^2+p-1)|E|^2-(2p^3+2p^2)|E|-p^4\bigg)\\
    &=\frac{(p^2+p-1)|E|^2}{p^3}\bigg(|E|-\frac{p^3+p^2+\sqrt{2p^6+3p^5}}{p^2+p-1}\bigg)\\
    &\times\bigg(|E|-\frac{p^3+p^2-\sqrt{2p^6+3p^5}}{p^2+p-1}\bigg).
\end{split}    
\end{equation*}

\medskip

Again, it suffices to show that the expression in the first parenthesis is positive.

\medskip

If $|E|>(\sqrt{3}+1)p$, then 
\begin{equation*}
    |E|-\frac{p^3+p^2+\sqrt{2p^6+3p^5}}{p^2+p-1}> p\bigg((\sqrt{3}+1)- \frac{(\sqrt{3}+1)p^3+p^2}{p^3+p^2-p}\bigg).
\end{equation*}

The function $\varphi:[3,+\infty)\to \mathbb{R}$ defined by $\varphi(x)=\mfrac{(\sqrt{3}+1)x^3+x^2}{x^3+x^2-x}$ decreases on $[3,x_0]$ and increases on $[x_0,+\infty),$ where $x_0\approx 3.32$. Since $\lim \limits_{x\to +\infty}\varphi(x)=\sqrt{3}+1$, then $\varphi(x)\leq \sqrt{3}+1$.
It means that the expression in the first parenthesis is positive and hence 
\begin{equation*}
    |S_1(r)|+|E|^2-2|E|^3>0. \qedhere   
\end{equation*}
\end{proof}

\medskip

Combining Lemma \ref{lemm3.1} with \eqref{eq3.8}, we obtain that if $|E|>(\sqrt{3}+1)p$, then 
\begin{equation*}
|\mathcal{C}|= \left|\left\{(x_1,x_2,x_3,y_1,y_2,y_3)\in E^6: \begin{array}{l}\lVert y_i-y_{i+1}\rVert=r\lVert x_i-x_{i+1}\rVert,\ i\in [2],\\
x_1\neq x_{2},\ x_2\neq x_3,\ x_1\neq x_3,\ y_1\neq y_3\end{array}\right\}\right|>0,
\end{equation*}

 which completes the proof of Theorem \ref{2chainsthm}.

\bigskip

\section{Proof of Theorem \ref{4cyclesthm}.}
\label{sec:4}

It is clear that $C(r)$ also contains pairs of 4-cycles in $E$ with dilation ratio $r$ which are degenerate. For example, degenerate pairs are pairs if one takes $(x_1=x_3)\lor (x_2=x_4)\lor (y_1=y_3)\lor (y_2=y_4)$. 

\medskip

Therefore, we need to rule out these degenerate cases and show that their size is less than the size of $C(r)$. That is why we will consider the following sets:  
\begin{equation*}
\mathcal{A}_{13}:= \left\{(x_1,x_2,x_3,x_4,y_1,y_2,y_3,y_4)\in E^8: \begin{array}{l}\lVert y_i-y_{i+1}\rVert=r\lVert x_i-x_{i+1}\rVert,\ x_i\neq x_{i+1},\ i\in [3],\\
\lVert y_4-y_{1}\rVert=r\lVert x_4-x_{1}\rVert,\ x_4\neq x_{1},\ x_1= x_3\end{array}\right\},
\end{equation*}    

\begin{equation*}
    \mathcal{A}_{24}:=\left\{(x_1,x_2,x_3,x_4,y_1,y_2,y_3,y_4)\in E^8: \begin{array}{l}\lVert y_i-y_{i+1}\rVert=r\lVert x_i-x_{i+1}\rVert,\ x_i\neq x_{i+1},\ i\in [3],\\
\lVert y_4-y_{1}\rVert=r\lVert x_4-x_{1}\rVert,\ x_4\neq x_{1},\ x_2= x_4\end{array}\right\},
\end{equation*}    

\begin{equation*}
\mathcal{B}_{13}:= \left\{(x_1,x_2,x_3,x_4,y_1,y_2,y_3,y_4)\in E^8: \begin{array}{l}\lVert y_i-y_{i+1}\rVert=r\lVert x_i-x_{i+1}\rVert,\ x_i\neq x_{i+1},\ i\in [3],\\
\lVert y_4-y_{1}\rVert=r\lVert x_4-x_{1}\rVert,\ x_4\neq x_{1},\ y_1= y_3\end{array}\right\},
\end{equation*}    

\begin{equation*}
    \mathcal{B}_{24}:= \left\{(x_1,x_2,x_3,x_4,y_1,y_2,y_3,y_4)\in E^8: \begin{array}{l}\lVert y_i-y_{i+1}\rVert=r\lVert x_i-x_{i+1}\rVert,\ x_i\neq x_{i+1},\ i\in [3],\\
\lVert y_4-y_{1}\rVert=r\lVert x_4-x_{1}\rVert,\ x_4\neq x_{1},\ y_2= y_4\end{array}\right\}.
\end{equation*}   

\medskip

We will also define the set which is exactly the family of pairs of 4-cycles in $E$ with dilation ratio $r$.

\medskip

\begin{equation*}
    \mathcal{F}:= \left\{(x_1,x_2,x_3,x_4,y_1,y_2,y_3,y_4)\in E^8: \begin{array}{l}\lVert y_i-y_{i+1}\rVert=r\lVert x_i-x_{i+1}\rVert,\ i\in [3],\\
    \lVert y_4-y_{1}\rVert=r\lVert x_4-x_{1}\rVert,\\
x_i\neq x_{j},\ y_i\neq y_j,\ i\neq j\in [4]\end{array}\right\}.
\end{equation*}   

\medskip

It is clear that we have the following set equality: 
\begin{equation*}
    C(r)=\mathcal{F}\sqcup (\mathcal{A}_{13}\cup \mathcal{A}_{24}\cup\mathcal{B}_{13}\cup \mathcal{B}_{24}).    
\end{equation*}

Hence, we have 
\begin{equation}
\label{eq4.1}    
    |C(r)|=|\mathcal{F}|+|\mathcal{A}_{13}\cup \mathcal{A}_{24}\cup\mathcal{B}_{13}\cup \mathcal{B}_{24}|.    
\end{equation}

\medskip

One can trivially estimate the term $|\mathcal{A}_{13}\cup \mathcal{A}_{24}\cup\mathcal{B}_{13}\cup \mathcal{B}_{24}|$ as follows:
\begin{equation}
\label{eq4.2}    
    |\mathcal{A}_{13}\cup \mathcal{A}_{24}\cup\mathcal{B}_{13}\cup \mathcal{B}_{24}|\leq |\mathcal{A}_{13}|+|\mathcal{A}_{24}|+|\mathcal{B}_{13}|+|\mathcal{B}_{24}| 
\end{equation}

Comparing \eqref{eq4.2} with \eqref{eq4.1}, we obtain the following lower bound for the size of $\mathcal{F}$:

\medskip

\begin{equation}
\label{eq4.3}    
    |\mathcal{F}|\geq |C(r)|-|\mathcal{A}_{13}|-|\mathcal{A}_{24}|-|\mathcal{B}_{13}|-|\mathcal{B}_{24}|.
\end{equation}

\medskip

Lemma \ref{lem2.4} gives us the lower bound for $|C(r)|$ and our current goal is to find the appropriate upper bound for the size of $\mathcal{A}_{13},\ \mathcal{A}_{24},\ \mathcal{B}_{13}$ and $\mathcal{B}_{24}$. We would like to point out that we have performed the same approach in the proof of Theorem \ref{2chainsthm} and we found the right the upper bounds for $|\mathcal{A}|$ and $|\mathcal{B}|$ relying on the trivial estimates provided by Lemma \ref{lem2.6}.

\smallskip

However, this approach is fruitless in the case of $4$-cycles. More precisely, if we apply Lemma \ref{lem2.6} to estimate the size of $\mathcal{A}_{ij}$ and $\mathcal{B}_{ij}$ we will not be able to get a nontrivial exponent for the size of $E$.

\smallskip

Fortunately, this barrier can be overcome if we know the arithmetic structure of the sphere $S_t$ in $\mathbb{F}_q^d$ and this information is given by Lemma \ref{lem2.5}.

\medskip

Since $p\equiv 3 \pmod 4$, then $-1$ is a quadratic nonresidue in $\mathbb{F}_p$ since  
\begin{equation*}\legendre{-1}{p}=(-1)^{\frac{p-1}{2}}=-1,
\end{equation*}where $\left(\dfrac{a}{b}\right)$ is a Legendre symbol. Since $d=2$, then Lemma \ref{lem2.5} gives us that 
\begin{equation}
\label{eq4.4}
    |S_t|=p+1\quad \mathrm{for} \quad t\in \mathbb{F}_p^{*}.
\end{equation}

The following Lemma gives us the correct
upper bound for the size of $\mathcal{A}_{13},\ \mathcal{A}_{24},\ \mathcal{B}_{13}$ and $\mathcal{B}_{24}$.

\smallskip

\begin{lem}
\label{lem4.2}
The following inequality holds:
\begin{equation}
\label{eq4.5}
    |S_{2}(r)|\leq |\mathcal{A}_{13}|,|\mathcal{A}_{24}|, |\mathcal{B}_{13}|,|\mathcal{B}_{24}|\leq (p+1)|S_{2}(r)|.
\end{equation}
\end{lem}

\begin{proof}

For concreteness, we will prove this inequality only for $|\mathcal{A}_{13}|$. The remaining inequalities can be proven in an analogous way. 

\smallskip

From the definition of $\mathcal{A}_{13}$ it follows that $|\mathcal{A}_{13}|=|\widehat{\mathcal{A}}_{13}|,$ where 

\begin{equation*}
    \widehat{\mathcal{A}}_{13}:= \left\{(x_1,x_2,x_4,y_1,y_2,y_3,y_4)\in E^7: \begin{array}{l}\lVert y_1-y_{2}\rVert=\lVert y_2-y_{3}\rVert=r\lVert x_1-x_{2}\rVert,\\ \lVert y_3-y_{4}\rVert=\lVert y_1-y_{4}\rVert=r\lVert x_1-x_{4}\rVert,\\
x_1\neq x_{2},\ x_1\neq x_4\end{array}\right\}.
\end{equation*}   

\smallskip

Consider the function $f: \widehat{\mathcal{A}}_{13}\to S_{2}(r)$ defined as 
\begin{equation*}
    (x_1,x_2,x_4,y_1,y_2,y_3,y_4)\xmapsto{f} (x_4,x_1,x_2,y_4,y_1,y_2).
\end{equation*}

\medskip

We notice that $f$ is surjective. Indeed, for each $(x,y,z,x',y',z')\in S_{2}(r)$ its preimage under $f$ is $(y,z,x,y',z',y',x')\in \widehat{\mathcal{A}}_{13}$. It immediately implies that $|S_{2}(r)|\leq|\widehat{\mathcal{A}}_{13}|$ and hence $|S_{2}(r)|\leq|\mathcal{A}_{13}|.$

\smallskip

Now we proceed to the proof of the RHS inequality in $\eqref{eq4.5}$. It suffices to show that for each $y\in S_{2}(r)$ the inequality $|f^{-1}(\{y\})|\leq p+1$ holds. Let's fix an arbitrary $y\in S_{2}(r)$, then $y=(y_1,y_2,y_3,y_4,y_5,y_6)$ and consider it's preimage, i.e. the set $f^{-1}(\{y\})$. We already know that $f^{-1}(\{y\})\neq \varnothing$ due to the surjectivity of $f$. Consider an element $x_0:=(y_2,y_3,y_1,y_5,y_6,y_5,y_4)$ and one can check that $x_0\in f^{-1}(\{y\})$. 

\smallskip

Choose arbitrary element $x\in f^{-1}(\{y\})$, then $x=(y_2,y_3,y_1,y_5,y_6,\alpha,y_4)$ with $\lVert\alpha-y_4 \rVert=r\lVert y_1-y_2\rVert$. We see that $\lVert y_1-y_2\rVert\neq 0$ since $x_0\in \widehat{\mathcal{A}}_{13}$. 

\smallskip

We have shown that for arbitrary $x\in f^{-1}(\{y\})$ we have $x-x_0=(\alpha-y_5)\cdot \vec{\mathbf{e}}_6$ with $\lVert \alpha-y_4\rVert=t$, where $\vec{\mathbf{e}}_6=(0,0,0,0,0,1,0)$ and $t:=r\lVert y_1-y_2\rVert\neq 0$. 

\smallskip

That can be written as the following set containment: 
\begin{equation}
\label{eq4.6}
    f^{-1}(\{y\})\subseteq \{x_0+\vec{\mathbf{e}}_6\cdot (\alpha-\pi_5(y)): \alpha\in \pi_4(y)+S_t\},
\end{equation}
where $\pi_j(y)$ is the $j$th coordinate of $y$ and $\pi_4(y)+S_t$ means the translation of sphere $S_t$ by $\pi_4(y)$. Containment \eqref{eq4.6} implies that 
\begin{equation}
\label{eq4.7}
    |f^{-1}(\{y\})|\leq |\{x_0+\vec{\mathbf{e}}_6\cdot (\alpha-\pi_5(y)): \alpha\in \pi_4(y)+S_t\}|.
\end{equation}

\smallskip

However, the RHS term in \eqref{eq4.7} is at most
\begin{equation}
\label{eq4.8}
    |\pi_4(y)+S_t|=|S_t|.
\end{equation}

\smallskip

Comparing \eqref{eq4.7} and \eqref{eq4.8} with \eqref{eq4.4}, we obtain 
\begin{equation*}
    |f^{-1}(\{y\})|\leq p+1,
\end{equation*}

which completes the proof of Lemma \ref{lem4.2}.
\end{proof}

\medskip

Plugging inequalities \eqref{eq4.5} to the inequality \eqref{eq4.3}, we obtain the following lower bound for the size of $\mathcal{F}$:
\begin{equation*}
    |\mathcal{F}|\geq |C(r)|-4(p+1)|S_{2}(r)|.
\end{equation*}

Applying Lemma \ref{lem2.4} we obtain 
\begin{equation} 
\label{eq4.9}
\begin{split}
    |\mathcal{F}| & \geq |E|^{-4}|S_{2}(r)|^2-4(p+1)|S_{2}(r)| \\[5pt]
 & = |E|^{-4}|S_{2}(r)|\left(|S_{2}(r)|-4(p+1)|E|^4 \right).
\end{split}
\end{equation}

\medskip

We shall show that $|\mathcal{F}|>0$ if $|E|\gg p^{\frac{3}{2}}$.

\medskip

\begin{lem}
\label{lem4.3}
If $|E|>4\sqrt{3}p^{\frac{3}{2}}$, then 
\begin{align*}
    |E|^{-4}|S_{2}(r)|>0 \quad \text{and} \quad |S_{2}(r)|-4(p+1)|E|^4>0.
\end{align*}
\end{lem} 

\begin{proof}
    Lemma \ref{lemm3.1} claims that if $|E|>(\sqrt{3}+1)p$, then $|E|^{-2}|S_1(r)|>0$. Applying Lemma $\ref{lower bound for S_2(r)}$, we obtain that $|E|^{-4}|S_{2}(r)|>0$.  

\medskip
    
    Lemma \ref{lower bound for S(r)} implies that 
    \begin{equation*}
        |S_1(r)|> \frac{|E|^4}{p}-\frac{2|E|^3}{p}-(p+1)|E|^2.
    \end{equation*}
    
    Therefore, we obtain
    \begin{equation}
    \label{eq4.10}
        |S_1(r)|> \frac{|E|^4}{3p},
    \end{equation}
    since \begin{equation*}
        \frac{|E|^4}{3p}>\frac{2|E|^3}{p}\quad \mathrm{and} \quad \frac{|E|^4}{3p}>(p+1)|E|^2    
    \end{equation*} provided that $|E|>(\sqrt{3}+1)p$. 
    
    Inequality \eqref{eq4.10} combined with Lemma \ref{lower bound for S_2(r)} implies that 
    \begin{align*}
        |S_{2}(r)|-4(p+1)|E|^4&>\frac{|E|^6}{9p^2}-4(p+1)|E|^4\\[5pt]
        &=\frac{|E|^4}{9p^2}\bigg(|E|^2-36p^2(p+1)\bigg).
    \end{align*}
    
    One can check that if $|E|>4\sqrt{3}p^{\frac{3}{2}}$, then $|E|^2>36p^2(p+1)$. Moreover, the first inequality $|E|^{-4}|S_{2}(r)|>0$ also holds if $|E|>4\sqrt{3}p^{\frac{3}{2}}$ since $4\sqrt{3}p^{\frac{3}{2}}>(\sqrt{3}+1)p$.  
\end{proof}

Combining Lemma \ref{lem4.3} with \eqref{eq4.9}, we obtain that if $|E|>4\sqrt{3}p^{\frac{3}{2}}$, then $|\mathcal{F}|>0$ which completes the proof of Theorem \ref{4cyclesthm}.

\bigskip

\section{Proof of Theorem \ref{trianglesthm}.}
\label{sec:5}

In this section we will obtain a nontrivial estimate for size of $E\subset \mathbb{F}_p^2$ such that it contains a pair of 3-cycles in $E$ with dilation ratio $r\in (\mathbb{F}_p)^2\setminus \{0\}$. The approach will be somewhat different and it relies on the introducing certain counting function and investigating its $L^3$-norm.

\medskip

If $r\in (\mathbb{F}_p)^2\setminus \{0\}$,\ $z\in \mathbb{F}_p^2$ and $\theta\in \mathrm{O}_2(\mathbb{F}_p)$, then consider the counting function defined as follows:
\begin{equation}
\label{eq5.1}
    \lambda_{r,\theta}(z):=|\{(u,v)\in E^2: u-\sqrt{r}\theta v=z\}|
\end{equation} and the cube of its $L^3$-norm 
\begin{equation*}
        \norm{\lambda_{r,\theta}(z)}_3^3:=\sum \limits_{\theta,z} \lambda^3_{r,\theta}(z).
\end{equation*}

One can see that 
\begin{equation*}
    \lambda^3_{r,\theta}(z)=|\{(u_1,u_2,u_3,v_1,v_2,v_3)\in E^6: u_i-\sqrt{r}\theta v_i=z,\ i\in [3]\}|.
\end{equation*}

Hence, we obtain
\begin{equation*}
\begin{split}
    \norm{\lambda_{r,\theta}(z)}_3^3&=\sum \limits_{\theta\in \mathrm{O}_2(\mathbb{F}_p)}\sum \limits_{z\in \mathbb{F}_p^2}|\{(u_1,u_2,u_3,v_1,v_2,v_3)\in E^6: u_i-\sqrt{r}\theta v_i=z,\ i\in [3]\}|\\[5pt]
    &=\sum \limits_{\theta\in \mathrm{O}_2(\mathbb{F}_p)}|\{(u_1,u_2,u_3,v_1,v_2,v_3)\in E^6: u_1-\sqrt{r}\theta v_1=u_2-\sqrt{r}\theta v_2=u_3-\sqrt{r}\theta v_3\}|.    
\end{split}    
\end{equation*}

If we introduce the following notation: \begin{equation*}
    \Lambda_{\theta}(r):=\left\{(u_1,u_2,u_3,v_1,v_2,v_3)\in E^6: u_1-\sqrt{r}\theta v_1=u_2-\sqrt{r}\theta v_2=u_3-\sqrt{r}\theta v_3\right\},
\end{equation*} 
then $\norm{\lambda_{r,\theta}(z)}_3^3=\sum \limits_{\theta}|\Lambda_{\theta}(r)|.$

Consider the following subset of $\Lambda_{\theta}(r)$, where $v_i$'s are pairwise distinct: 
\begin{equation*}
    N_{\theta}(r):= \left\{(u_1,u_2,u_3,v_1,v_2,v_3)\in E^6: \begin{array}{l} u_i-u_j=\sqrt{r}\theta (v_i-v_j),\\
v_i\neq v_j,\ i\neq j\in [3] \end{array}\right\}.
\end{equation*}   

Applying inclusion-exclusion principle one can compute the size of $N_{\theta}(r)$ explicitly. Indeed,

\begin{equation}
\label{eq5.2}
\begin{split}
    |N_{\theta}(r)|&=|\Lambda_{\theta}(r)|-\sum \limits_{1\leq k<l\leq 3}\left|\left\{(u_1,u_2,u_3,v_1,v_2,v_3)\in E^6: \begin{array}{l} u_i-u_j=\sqrt{r}\theta (v_i-v_j),\\
i\neq j\in [3],\ v_k=v_l \end{array}\right\} \right|\\[5pt]
&+2\left|\left\{(u_1,u_2,u_3,v_1,v_2,v_3)\in E^6: \begin{array}{l} u_i-u_j=\sqrt{r}\theta (v_i-v_j),\ i\neq j\in [3],\\
v_1=v_2=v_3  \end{array}\right\} \right|.
\end{split}
\end{equation}

One can show that 

\smallskip

\begin{equation}
\label{eq5.3}
    \left|\left\{(u_1,u_2,u_3,v_1,v_2,v_3)\in E^6: \begin{array}{l} u_i-u_j=\sqrt{r}\theta (v_i-v_j),\\
i\neq j\in [3],\ v_k=v_l \end{array}\right\} \right|\\=\sum \limits_{z\in \mathbb{F}_p^2}\lambda_{r,\theta}^2(z).
\end{equation} 

We will prove $\eqref{eq5.3}$ only for $(k,l)=(1,2)$ since the remaining two cases can be done in an analogous way. 

\medskip
    
Indeed, if $(k,l)=(1,2)$, then
\begin{equation*}
\begin{split}
\left|\left\{(u_1,u_2,u_3,v_1,v_2,v_3)\in E^6: \begin{array}{l} u_i-u_j=\sqrt{r}\theta (v_i-v_j),\\[5pt]
i\neq j\in [3],\ v_1=v_2 \end{array}\right\} \right|\\[5pt]
\end{split}    
\end{equation*}

\begin{equation*}
\begin{split}
    &=\left|\left\{(u_1,u_3,v_1,v_3)\in E^4: u_1-\sqrt{r}\theta v_1=u_3-\sqrt{r}\theta v_3\right\}\right|\\[5pt]
    &=\sum_{z\in \mathbb{F}_p^2}|\{(u_1,u_3,v_1,v_3)\in E^4: u_1-\sqrt{r}\theta v_1=u_3-\sqrt{r}\theta v_3=z\}|\\
\end{split}
\end{equation*}

\begin{equation*}
\begin{split}
    =\sum_{z\in \mathbb{F}_p^2}|\{(u_1,v_1)\in E^2: u_1-\sqrt{r}\theta v_1&=z\}|\times |\{(u_3,v_3)\in E^2: u_3-\sqrt{r}\theta v_3=z\}|\\
    &=\sum \limits_{z\in \mathbb{F}_p^2}\lambda_{r,\theta}^2(z).
\end{split}
\end{equation*}

One can verify that 

\smallskip

\begin{equation}
\label{eq5.4}
\left|\left\{(u_1,u_2,u_3,v_1,v_2,v_3)\in E^6: \begin{array}{l} u_i-u_j=\sqrt{r}\theta (v_i-v_j),\ i\neq j\in [3],\\
v_1=v_2=v_3  \end{array}\right\} \right|=|E|^2.
\end{equation}

\smallskip

Therefore, if we plug $\eqref{eq5.3}$ and $\eqref{eq5.4}$ into $\eqref{eq5.2}$, we obtain
\begin{equation}
\label{eq5.5}
    |N_{\theta}(r)|=|\Lambda_{\theta}(r)|-3\sum \limits_{z\in \mathbb{F}_p^2}\lambda_{r,\theta}^2(z)+2|E|^2.
\end{equation}

Summing $\eqref{eq5.5}$ over all $\theta \in \mathrm{O}_2(\mathbb{F}_p)$, we obtain
\begin{equation}
\label{eq5.6}
\begin{split}
    \sum \limits_{\theta}|N_{\theta}(r)|&=\sum \limits_{\theta}|\Lambda_{\theta}(r)|-3\sum \limits_{\theta,z}\lambda_{r,\theta}^2(z)+2|E|^2\times |\mathrm{O}_2(\mathbb{F}_p)|\\[5pt]
    &=\sum \limits_{\theta,z}\lambda_{r,\theta}^3(z)-3\sum \limits_{\theta,z}\lambda_{r,\theta}^2(z)+2|E|^2\times |\mathrm{O}_2(\mathbb{F}_p)|.
\end{split}
\end{equation}

One can check that for arbitrary $\theta\in \mathrm{O}_2(\mathbb{F}_p)$ the following set containment holds: 
\begin{equation}
\label{eq5.7}
    N_{\theta}(r)\subset \mathcal{T}, 
\end{equation}
where 
\begin{equation*}
\mathcal{T}:=\left\{(u_1,u_2,u_3,v_1,v_2,v_3)\in E^6: \begin{array}{l} \lVert u_i-u_j\rVert=r\lVert v_i-v_j\rVert,\ v_i\neq v_j,\\
u_i\neq u_j,\ i\neq j\in [3]  \end{array}\right\}.
\end{equation*}

One can see that $\mathcal{T}$ is exactly the family of pairs of 3-cycles in $E$ with dilation ratio $r\in (\mathbb{F}_p)^2\setminus \{0\}$.

\medskip

Containment $\eqref{eq5.7}$ implies that

\begin{equation*}
|\mathcal{T}|\geq \frac{1}{|\mathrm{O}_2(\mathbb{F}_p)|}\sum \limits_{\theta}|N_{\theta}(r)|. 
\end{equation*}

Applying $\eqref{eq5.6}$, we obtain 

\begin{equation}
\label{eq5.8}
    |\mathcal{T}|\geq \frac{1}{|\mathrm{O}_2(\mathbb{F}_p)|}\bigg(\sum \limits_{\theta,z}\lambda_{r,\theta}^3(z)-3\sum \limits_{\theta,z}\lambda_{r,\theta}^2(z)\bigg).
\end{equation}

In other words, we managed to obtain the lower bound for the size of $\mathcal{T}$ in terms of $L_2,L_3$-norms of $\lambda_{r,\theta}(z)$ and the size of $\mathrm{O}_2(\mathbb{F}_p)$. More precisely, inequality $\eqref{eq5.8}$ can be rewritten in the following equivalent way:
\begin{equation}
\label{eq5.9}
    |\mathcal{T}|\geq \frac{1}{|\mathrm{O}_2(\mathbb{F}_p)|}\bigg(\norm{\lambda_{r,\theta}(z)}_3^3-3\norm{\lambda_{r,\theta}(z)}_2^2\bigg).
\end{equation}

\smallskip

It remains to show that the RHS in \eqref{eq5.9} is positive. 
The lower bound for $L_3$-norm of $\lambda_{r,\theta}(z)$ can be obtained by means of Hölder's inequality. Indeed,

\begin{equation}
\label{eq5.10}
    \sum \limits_{z\in \mathbb{F}_p^2}\lambda_{r,\theta}(z)\leq \bigg(\sum \limits_{z\in \mathbb{F}_p^2}\lambda_{r,\theta}^3(z)\bigg)^{\frac{1}{3}}\times \bigg(\sum \limits_{z\in\mathbb{F}_p^2}1 \bigg)^{\frac{2}{3}}.
\end{equation}

From the definition of $\lambda_{r,\theta}(z)$ it follows that $\sum \limits_{z\in \mathbb{F}_p^2}\lambda_{r,\theta}(z)=|E|^2$ and taking this into account we can rewrite $\eqref{eq5.10}$ in the following way: 
\begin{equation}
\label{eq5.11}
    \sum\limits_{z\in \mathbb{F}_p^2}\lambda_{r,\theta}^3(z)\geq \frac{|E|^6}{p^4}.
\end{equation}

Summing inequality $\eqref{eq5.11}$ over all $\theta\in \mathrm{O}_2(\mathbb{F}_p)$, we obtain the following lower bound for $L_3$-norm: 

\begin{equation}
\label{eq5.12}
    \norm{\lambda_{r,\theta}(z)}_3^3=\sum \limits_{\theta,z}\lambda_{r,\theta}^3(z)\geq \frac{|E|^6}{p^3},
\end{equation} since $|\mathrm{O}_2(\mathbb{F}_p)|\geq p$. 

\medskip

The inequality \eqref{eq5.9} can be written as follows:
\begin{equation}
\label{eq5.13}
\begin{split}
    |\mathcal{T}|&\geq \frac{1}{|\mathrm{O}_2(\mathbb{F}_p)|}\bigg(\sum \limits_{\theta,z}\lambda_{r,\theta}^3(z)-3\sum \limits_{\theta,z}\lambda_{r,\theta}^2(z)\bigg)\\
    &=\frac{1}{|\mathrm{O}_2(\mathbb{F}_p)|}\bigg(\underbrace{\frac{1}{2}\sum \limits_{\theta,z}\lambda_{r,\theta}^3(z)-3\sum \limits_{ \lambda\geq 6}\lambda_{r,\theta}^2(z)}_{=\textup{I}}+\underbrace{\frac{1}{2}\sum \limits_{\theta,z}\lambda_{r,\theta}^3(z)-3\sum \limits_{\lambda <6}\lambda_{r,\theta}^2(z)}_{=\textup{II}}\bigg).
\end{split}
\end{equation}

We see that \textup{I} is nonnegative. Indeed,

\begin{equation*}
    \textup{I}\geq \frac{1}{2}\sum_{\lambda\geq 6}\lambda^3_{r,\theta}(z)-3\sum_{\lambda\geq 6}\lambda^2_{r,\theta}(z)=\sum_{\lambda \geq 6}\frac{\lambda^2_{r,\theta}(z)}{2}\big(\lambda_{r,\theta}(z)-6\big)\geq 0.
\end{equation*}

Moreover, we have the following lower bound for $\textup{II}$:
\begin{equation*}
\begin{split}
   \textup{II}&\geq \frac{1}{2}\sum_{\theta,z}\lambda^3_{r,\theta}(z)-108\sum_{\lambda<6}1\\
   &\geq \frac{1}{2}\sum_{\theta,z}\lambda^3_{r,\theta}(z)-108\times |\mathbb{F}_p^2|\times |\mathrm{O}_2(\mathbb{F}_p)|\\
   &\geq \frac{|E|^6}{2p^3}-324p^3.
\end{split}
\end{equation*}

\smallskip

Combining lower bounds for \textup{I} and \textup{II}, we obtain
\begin{equation}
\label{eq5.14}
    \textup{I}+\textup{II}\geq \frac{|E|^6}{2p^3}-324p^3.
\end{equation}

We know that $|\mathrm{O}_2(\mathbb{F}_p)|>0$, then combining equations \eqref{eq5.13} and \eqref{eq5.14}, we obtain 
\begin{equation*}
    |\mathcal{T}|\geq \frac{1}{|\mathrm{O_2}(\mathbb{F}_p)|}\Bigg(\frac{|E|^6}{2p^3}-324p^3\Bigg).
\end{equation*}

It implies that if $|E|\geq 3p$, then $\mfrac{|E|^6}{2p^3}-324p^3>0$ and thus $|\mathcal{T}|>0$ and it completes the proof of Theorem \ref{trianglesthm}.

\bigskip

\section{Proof of Theorem \ref{ptsconfigthm}.}
\label{sec:6}

In this section we will extend Theorem \ref{trianglesthm} to the case of $d$-simplexes. The approach will be analogous to the one which we implemented in the course of the proof of Theorem \ref{trianglesthm} but with minor modifications. We will investigate the size of $\mathrm{O}_d(\mathbb{F}_p)$ and $L^{d+1}$-norm of the counting function $\lambda_{r,\theta}(z)$ which was defined in $\eqref{eq5.1}$.

\medskip

We will show that if $E\subset \mathbb{F}_p^d$ and  $|E|\gg_d p^{\frac{d}{2}}$, then $|\mathcal{P}|>0$, where
\begin{equation*}
    \mathcal{P}:= \left\{(u_1,\dots,u_{d+1},v_1,\dots,v_{d+1})\in E^{2d+2}: \begin{array}{l}\lVert u_i-u_{j}\rVert=r\lVert v_i-v_j\rVert,\ v_i\neq v_j,\\ 
u_i\neq u_j,\ i\neq j\in [d+1]\end{array}\right\}.
\end{equation*}

For the counting function introduced in $\eqref{eq5.1}$ we consider it's $L^{d+1}$-norm:

\begin{gather} 
\label{eq6.1}
\lVert\lambda_{r,\theta}(z)\rVert_{d+1}^{d+1}:=\sum_{\theta,z}\lambda_{r,\theta}^{d+1}(z) \\[5pt]
\nonumber
\begin{aligned}
&=\sum_{\theta,z}
  \left|\left\{
  (u_1,\dots,u_{d+1},v_1,\dots,v_{d+1}) \in E^{2d+2}: 
   u_1-\sqrt{r}\theta v_1 = \dots = u_{d+1}-\sqrt{r}\theta v_{d+1} = z\right\} \right|\\[5pt]
&=\sum_{\theta}
  \left|\left\{
  (u_1,\dots,u_{d+1},v_1,\dots,v_{d+1}) \in E^{2d+2}: 
  u_i-u_j=\sqrt{r}\theta(v_i-v_j),\ 
  1\leq i<j\leq d+1
  \right\}\right|. 
\end{aligned}
\end{gather}

We let $\Lambda_{\theta}(r)$ denote the set 
\begin{equation*}
     \left\{(u_1,\dots,u_{d+1},v_1,\dots,v_{d+1})\in E^{2d+2}: \begin{array}{l} u_i-u_{j}=\sqrt{r}\theta(v_i-v_j),\\ 
1\leq i<j\leq d+1\end{array}\right\}.
\end{equation*}

Therefore, \eqref{eq6.1} can be rewritten in the following way: 
\begin{equation}
\label{eq6.2}
    \norm{\lambda_{r,\theta}(z)}_{d+1}^{d+1}:=\sum \limits_{\theta}|\Lambda_{\theta}(r)|.
\end{equation}

In $\Lambda_{\theta}(r)$ we extract the subset where $v_i\neq v_j$ for $i\neq j$.
\begin{equation*}
     N_{\theta}(r):=\left\{(u_1,\dots,u_{d+1},v_1,\dots,v_{d+1})\in E^{2d+2}: \begin{array}{l} u_i-u_{j}=\sqrt{r}\theta(v_i-v_j),\ v_i\neq v_j,\\ 
i\neq j\in [d+1]\end{array}\right\}.
\end{equation*}

\begin{remark}
One can check that if $\mathbf{u}=\sqrt{r}\theta \mathbf{v}$ for $\theta\in \mathrm{O}_d(\mathbb{F}_p)$, then $\lVert \mathbf{u}\rVert=r\lVert \mathbf{v}\rVert$.
\end{remark}

This remark immediately implies that for each $\theta \in \mathrm{O}_d(\mathbb{F}_p)$ we have $N_{\theta}(r)\subset \mathcal{P}.$

\medskip

Hence, we have 
\begin{equation}
\label{eq6.3}
    |\mathcal{P}|\geq \frac{1}{|\mathrm{O}_d(\mathbb{F}_p)|}\sum \limits_{\theta}|N_{\theta}(r)|.
\end{equation}

For each pair $(k,l)$ such that $1\leq k<l\leq d+1$, we define the following set: 
\smallskip
\begin{equation*}
     A_{kl}:=\left\{(u_1,\dots,u_{d+1},v_1,\dots,v_{d+1})\in E^{2d+2}: \begin{array}{l} u_i-u_{j}=\sqrt{r}\theta(v_i-v_j),\\ 
1\leq i<j\leq d+1,\ v_k=v_l\end{array}\right\}.
\end{equation*}

\smallskip

One can see that the following set equality holds:
\begin{equation*}
    \Lambda_{\theta}(r)\setminus \bigcup_{1\leq k<l\leq d+1}A_{kl}=N_{\theta}(r).
\end{equation*}

Applying Bonferroni inequality, we obtain
\smallskip
\begin{equation}
\label{eq6.4}
    |N_{\theta}(r)|\geq |\Lambda_{\theta}(r)|-\sum \limits_{1\leq k<l\leq d+1}|A_{kl}|.
\end{equation}

One can show that for each such pair $(k,l)$, we have 
\begin{equation}
\label{eq6.5}
    |A_{kl}|=\sum \limits_{z\in \mathbb{F}_p^d}\lambda_{r,\theta}^d(z).
\end{equation}

Plugging $\eqref{eq6.5}$ into inequality $\eqref{eq6.4}$, we obtain 
\begin{equation}
\label{eq6.6}
    |N_{\theta}(r)|\geq |\Lambda_{\theta}(r)|-\binom{d+1}{2}\sum \limits_{z\in \mathbb{F}_p^d}\lambda_{r,\theta}^d(z).
\end{equation}

Summing $\eqref{eq6.6}$ over all $\theta\in \mathrm{O}_d(\mathbb{F}_p)$, we obtain the following inequality:
\medskip
\begin{equation}
\label{eq6.7}
    \sum \limits_{\theta}|N_{\theta}(r)|\geq \sum \limits_{\theta}|\Lambda_{\theta}(r)|-\binom{d+1}{2}\sum \limits_{\theta,z}\lambda_{r,\theta}^d(z).
\end{equation}
\medskip

Taking into account $\eqref{eq6.2}$ and $\eqref{eq6.7}$ the inequality $\eqref{eq6.3}$ can be written as 
\begin{equation}
\label{eq6.8}
    |\mathcal{P}|\geq \frac{1}{|\mathrm{O}_d(\mathbb{F}_p)|}\bigg(\norm{\lambda_{r,\theta}(z)}_{d+1}^{d+1}-\binom{d+1}{2}\norm{\lambda_{r,\theta}(z)}_{d}^{d}\bigg).
\end{equation}

We can derive the lower bound for the $\norm{\lambda_{r,\theta}(z)}_{d+1}^{d+1}$ by means of Hölder's inequality. Indeed,

\begin{equation}
\label{eq6.9}
    \sum \limits_{z\in \mathbb{F}_p^d}\lambda_{r,\theta}(z)\leq \bigg(\sum \limits_{z\in \mathbb{F}_p^d}\lambda_{r,\theta}^{d+1}(z)\bigg)^{\frac{1}{d+1}}\times \bigg(\sum \limits_{z\in \mathbb{F}_p^d}1\bigg)^{\frac{d}{d+1}}. 
\end{equation}

From the definition of $\lambda_{r,\theta}(z)$ follows that $\sum \limits_{z\in \mathbb{F}_p^d}\lambda_{r,\theta}(z)=|E|^2.$ 

Therefore, $\eqref{eq6.9}$ implies that 
\begin{equation}
\label{eq6.10}
\sum \limits_{z\in \mathbb{F}_p^d}\lambda_{r,\theta}^{d+1}(z)\geq \frac{|E|^{2d+2}}{p^{d^2}}.
\end{equation}

Summing $\eqref{eq6.10}$ over all $\theta \in \mathrm{O}_d(\mathbb{F}_p)$, we obtain the following lower bound for the $L^{d+1}$-norm of $\lambda_{r,\theta}(z)$:

\begin{equation}
\label{eq6.11}
    \norm{\lambda_{r,\theta}(z)}_{d+1}^{d+1}\geq |\mathrm{O}_d(\mathbb{F}_p)|\times \frac{|E|^{2d+2}}{p^{d^2}}.
\end{equation}

We can write the inequality \eqref{eq6.8} as follows:
\begin{equation*}
    |\mathcal{P}|\geq \frac{1}{|\mathrm{O}_d(\mathbb{F}_p)|}\Bigg(\norm{\lambda_{r,\theta}(z)}_{d+1}^{d+1} -\binom{d+1}{2}\Bigg(\sum_{\lambda \geq d(d+1)}\lambda^{d}_{r,\theta}(z)+\sum_{\lambda<d(d+1)}\lambda^{d}_{r,\theta}(z)\Bigg)\Bigg).
\end{equation*}

\medskip

We notice that the expression in the parenthesis can be written as $\textup{I}+\textup{II}$, where

\begin{equation*}
    \textup{I}\coloneqq \frac{\norm{\lambda_{r,\theta}(z)}_{d+1}^{d+1}}{2} -\binom{d+1}{2}\sum_{\lambda \geq d(d+1)}\lambda^{d}_{r,\theta}(z),
\end{equation*}
\begin{equation*}
    \textup{II}\coloneqq \frac{\norm{\lambda_{r,\theta}(z)}_{d+1}^{d+1}}{2} -\binom{d+1}{2}\sum_{\lambda < d(d+1)}\lambda^{d}_{r,\theta}(z).
\end{equation*}

It is not difficult to show that $\textup{I}\geq 0$. Indeed,

\begin{equation}
\label{eq6.12}
\begin{split}
    \textup{I}&=\frac{1}{2}\sum_{\theta, z}\lambda_{r,\theta}^{d+1}(z)-\binom{d+1}{2}\sum_{\lambda\geq d(d+1)}\lambda_{r,\theta}^d(z)\\[5pt]
    &\geq\frac{1}{2}\sum_{\lambda\geq d(d+1)}\lambda_{r,\theta}^{d+1}(z)-\binom{d+1}{2}\sum_{\lambda\geq d(d+1)}\lambda_{r,\theta}^d(z)\\[5pt]
    &=\frac{1}{2}\sum_{\lambda\geq d(d+1)} \lambda_{r,\theta}^{d}(z)\Big(\lambda_{r,\theta}(z)-d(d+1) \Big)\geq 0.
\end{split}
\end{equation}

Using \eqref{eq6.11} one can obtain the following lower bound for \textup{II}:

\medskip

\begin{equation}
\label{eq6.13}
\begin{split}
    \textup{II}&\geq  \frac{|\mathrm{O}_d(\mathbb{F}_p)|}{2}\times \frac{|E|^{2d+2}}{p^{d^2}}-(d(d+1))^d\binom{d+1}{2}\times|\mathrm{O}_d(\mathbb{F}_p)|\times|\mathbb{F}_p^d|\\[5pt]
    &=\frac{|\mathrm{O}_d(\mathbb{F}_p)|}{2}\Bigg(\frac{|E|^{2d+2}}{p^{d^2}}-(d^2+d)^{d+1}p^d \Bigg).
\end{split}   
\end{equation}

Combining inequalities \eqref{eq6.12}, \eqref{eq6.13} and since $|\mathrm{O}_2(\mathbb{F}_p)|>0$, we have
\begin{equation*}
    |\mathcal{T}|\geq \frac{1}{2}\Bigg(\frac{|E|^{2d+2}}{p^{d^2}}-(d^2+d)^{d+1}p^d\Bigg).
\end{equation*}

It is easy to verify that if $|E|\geq (d+1)p^{\frac{d}{2}}$, then $|\mathcal{T}|>0$ and this completes the proof of Theorem \ref{ptsconfigthm}.

\bigskip

\section{Proof of Theorem \ref{lowerboundS_3(r)}}
\label{sec:7}

Basically, using Cauchy-Schwarz inequality we were able to derive lower bounds for $|S_{2}(r)|$ and $|C(r)|$ in Lemma \ref{lower bound for S_2(r)} and \ref{lem2.4} in terms of $|S_1(r)|$. However, this technique fails for $|S_3(r)|$ since paths of length 3 have an odd number of edges. In this section, we will show how to obtain the lower bound for $|S_k(r)|$ in terms of $|S_1(r)|$ by means of Theorem \ref{3path}. 

\smallskip

Let $r\in \mathbb{F}_p^{*}$, $p$ be a prime such that $p\equiv 3 \pmod 4$ and $E\subset \mathbb{F}_p^2$.

\smallskip

Define the graph $G=(V,E)$ as follows: let $V:=E\times E\equiv\{(x,x'):x,x'\in E\}$. If $(x,x'),(y,y')\in V$, then we connect them via an edge iff $(x,x')\neq (y,y')$ and $\lVert y'-x'\rVert=r\lVert y-x\rVert$.

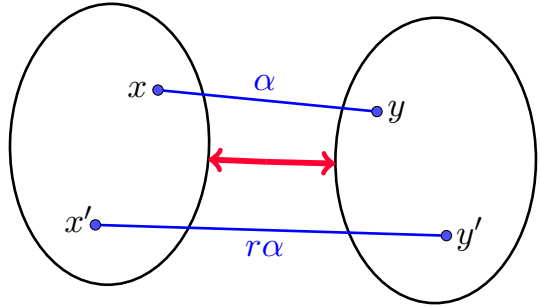
\begin{figure}[htbp]
\centering
\usetikzlibrary{arrows}

\definecolor{ffqqtt}{rgb}{1.,0.,0.2}
\definecolor{ududff}{rgb}{0.30196078431372547,0.30196078431372547,1.}
\definecolor{qqqqff}{rgb}{0.,0.,1.}

\begin{tikzpicture}[scale=0.65][line cap=round,line join=round,>=triangle 45,x=1.0cm,y=1.0cm]
\clip(-1.,-5.5) rectangle (13.5,2.);
\draw [rotate around={88.58905917961856:(2.93,-1.57)},line width=1.pt] (2.93,-1.57) ellipse (2.875334293907589cm and 2.0357178836275556cm);
\draw [rotate around={89.44909602078515:(9.6,-1.9)},line width=1.pt] (9.6,-1.9) ellipse (2.921090065348543cm and 2.0508454768407027cm);
\draw [line width=1.pt,color=qqqqff] (3.92,-0.46)-- (8.4,-0.9);
\draw [line width=1.pt,color=qqqqff] (2.64,-3.22)-- (9.82,-3.44);
\draw [->,line width=2.2pt,color=ffqqtt] (6.133333333333335,-1.9288888888888882) -- (7.5492881733743875,-1.964135623850796);
\draw [->,line width=2.2pt,color=ffqqtt] (6.133333333333335,-1.9288888888888882) -- (4.950241370090094,-1.8811883012738628);

\draw[color=qqqqff] (6.1,-0.3) node {\scalebox{1.2}{$\alpha$}};
\draw[color=qqqqff] (6.1,-3.7) node {\scalebox{1.2}{$r\alpha$}};

\draw [fill=ududff] (3.92,-0.46) circle (3pt);
\draw[color=black] (3.5,-0.45) node { \scalebox{1.2} {$x$} };
\draw [fill=ududff] (8.4,-0.9) circle (3pt);
\draw[color=black] (8.8,-0.9) node { \scalebox{1.2} {$y$}};
\draw [fill=ududff] (2.64,-3.22) circle (3pt);
\draw[color=black] (2.3,-3.1) node { \scalebox{1.2} {$x'$}};
\draw [fill=ududff] (9.82,-3.44) circle (3pt);
\draw[color=black] (10.3,-3.4) node 
{\scalebox{1.2} {$y'$} };
\end{tikzpicture}
\caption{Joining vertices $(x,x')$ and $(y,y')$ with an edge.}
\end{figure}

Fix $(x,x')\in V$ and consider the degree of this vertex:
\begin{equation*}
\begin{split}
    \deg((x,x'))&=\left|\left\{(y,y')\in E^2: (y,y') \ \mathrm{is\ incident\ with}\ (x,x')\right\}\right|\\[5pt]
    &=\left|\left\{(y,y')\in E^2: \lVert y'-x'\rVert=r\lVert y-x\rVert, \ (x,x')\neq (y,y')\right\}\right|.
\end{split}
\end{equation*}

Therefore, we have \begin{equation*}
\begin{split}
    \sum \limits_{(x,x')\in V}\deg((x,x'))&=\sum  \limits_{x,x'\in E}\left|\left\{(y,y')\in E^2: \lVert y'-x'\rVert=r\lVert y-x\rVert, \ (x,x')\neq (y,y')\right\}\right|\\[5pt]
    &=\left|\left\{(x,y,x',y')\in E^4: \lVert y'-x'\rVert=r\lVert y-x\rVert, \ (x,x')\neq (y,y')\right\}\right|\\[5pt]
    &=\left|\left\{(x,y,x',y')\in E^4: \lVert y'-x'\rVert=r\lVert y-x\rVert, \ x\neq y\right\}\right|\\[5pt]
    &=|S_1(r)|.
\end{split}
\end{equation*}
\begin{remark}
In the penultimate equality we have used the fact that $\lVert x\rVert=0$ iff $x=(0,0)$ since $p\equiv 3 \pmod 4$.
\end{remark}

By degree sum formula it follows that $e(G)=|S_1(r)|/{2}$. Therefore, we constructed the graph with $|E|^2$ vertices and ${|S_1(r)|}/{2}$ edges.

\medskip

The number of paths of length $k\geq 3$ in our graph $G=(V,E)$ is equal to the size of $S_k(r)$. Indeed,
\begin{equation*}
    \left|\left\{(v_1,v_2,\dots,v_{k+1})\in V^{k+1}:v_iv_{i+1}\in E,\ i\in [k]\right\}\right|
\end{equation*}

\begin{equation*}
    =\left|\left\{((x_1,y_1),(x_2,y_2),\dots,(x_{k+1},y_{k+1}))\in V^{k+1}: (x_i,y_i)\sim(x_{i+1},y_{i+1}),\ i\in [k]\right\}\right|
\end{equation*}

\begin{equation*}
    =\left|\left\{(x_1,\dots,x_{k+1},y_1,\dots,y_{k+1})\in E^{2k+2}: \lVert y_i-y_{i+1}\rVert=r\lVert x_i-x_{i+1}\rVert,\ x_i\neq x_{i+1},\ i\in [k] \right\}\right|
\end{equation*}
\begin{equation*}
    =|S_k(r)|.
\end{equation*}

Combining this with Theorem \ref{3path}, we have 
\begin{equation}
\label{eq7.4}
    |S_k(r)|\geq \frac{(2e(G))^k}{n^{k-1}} \quad \Leftrightarrow \quad |S_k(r)|\geq \frac{|S_1(r)|^k}{|E|^{2k-2}}.    
\end{equation}

Lemma \ref{lower bound for S(r)} implies that if $|E|>2p$, then 
\begin{equation}
\label{weaklowerboundS_1(r)}
    |S_1(r)|> \dfrac{|E|^4}{3p}.     
\end{equation}

Therefore, comparing $\eqref{eq7.4}$ with $\eqref{weaklowerboundS_1(r)}$, we obtain 
\begin{equation*}
    |S_k(r)|>\frac{|E|^{2k+2}}{(3p)^k},
\end{equation*} 
which proves Theorem $\ref{lowerboundS_3(r)}$.

\bibliographystyle{plain}
\bibliography{refref.bib}

\end{document}